\documentclass[11pt, reqno]{amsart}

\usepackage[
    a4paper,
    top=32mm,
    bottom=32mm,
    left=25mm,
    right=25mm
]{geometry}

\usepackage{amsmath,amsthm,amsfonts,amscd,amssymb}
\usepackage{graphicx,color,bm,mathtools}

\usepackage[colorlinks=true, citecolor=blue, filecolor=black, linkcolor=black, urlcolor=black]{hyperref}

\usepackage{caption}
\usepackage{subcaption}

\usepackage[T1]{fontenc}
\usepackage{lmodern}

\title[{\fontsize{7pt}{7pt}\selectfont Smooth multi-trace statistics.}]
{Smooth Multi-Trace Statistics of Classical Ensembles: Large $N$ Expansions, Cumulants, and Matrix Integrals.}
\author {Beno\^\i{}t Collins}
\address{Department of Mathematics, Kyoto University} 
\email{collins@math.kyoto-u.ac.jp}
\author {Manasa Nagatsu}
\address{Department of Mathematics, Kyoto University} 
\email{nagatsu.manasa.64s@st.kyoto-u.ac.jp}

\numberwithin{equation}{section}

\theoremstyle{plain}
\newtheorem{lemma}{Lemma}[section]
\newtheorem{theorem}[lemma]{Theorem}
\newtheorem{proposition}[lemma]{Proposition}

\theoremstyle{definition}
\newtheorem{definition}[lemma]{Definition}

\theoremstyle{remark}
\newtheorem{remark}[lemma]{Remark}

\DeclareMathOperator{\tr}{tr} 
\DeclareMathOperator{\Tr}{Tr}
\DeclareMathOperator{\Wg}{Wg}

\DeclareMathOperator{\Moeb}{Moeb}
\DeclareMathOperator{\U}{U}
\DeclareMathOperator{\Ort}{O}

\DeclareMathOperator{\Part}{Part}
\DeclareMathOperator{\Sp}{Sp}

\newcommand{\E}{{\mathbb{E}}}
\renewcommand{\P}{\mathbb{P}}
\newcommand{\C}{{\mathbb{C}}}
\newcommand{\R}{{\mathbb{R}}}
\newcommand{\Z}{{\mathbb{Z}}}
\newcommand{\N}{{\mathbb{N}}}
\renewcommand{\d}{\mathrm{d}}
\newcommand{\e}{\mathrm{e}}

\normalsize

\begin{document}

\begin{abstract}
We consider expectations of the form $\E[\tr h_1(X_1^N) \cdots \tr h_r(X_r^N)]$, where $X_i^N$ are self-adjoint polynomials in various independent classical random matrices and $h_i$ are smooth test function and obtain a large $N$ expansion of these quantities, building on the framework of polynomial approximation and Bernstein-type inequalities recently developed by Chen, Garza-Vargas, Tropp, and van Handel.
As applications of the above, we prove the higher-order asymptotic vanishing of cumulants for smooth linear statistics, establish a Central Limit Theorem, and demonstrate the existence of formal asymptotic expansions for the free energy and observables of matrix integrals with smooth potentials.
\end{abstract}

\maketitle

\tableofcontents

\section{Introduction}
The asymptotic behavior of large random matrices as the dimension $N$ tends to infinity is a central subject in modern probability theory and mathematical physics. 
A fundamental feature of many matrix ensembles is the existence of a topological expansion (or $1/N$ expansion) for the expectation of the traces. 
Specifically, for these random matrices $X_1^N, \dots, X_d^N$, quantities of the form $\E [\tr P(X^N)]$ often admit a power series expansion in $N^{-1}$ (or $N^{-2}$ for GUEs or Haar-distributed unitary matrices because of the symmetry), where the coefficients relate to the enumeration of graphs. 
This connection between matrix models and physical or statistical models has been studied, for example, in \cite{tHooft74_MR413809, BIPZ78_MR471676, HarerZagier86_MR848681}.

The existence of these expansions is well-established for \textit{polynomial} test functions. 
In the polynomial setting, these expansions have been mainly derived using combinatorial techniques, including the Weingarten calculus, which provides a systematic way to compute integrals over compact groups such as the unitary group. 
The theory of Weingarten calculus was first developed in the theoretical physics context \cite{Weingarten}, and later formalized mathematically, see for example \cite{Collins2003,CS06_MR2217291}.

However, extending these results to broader classes of test functions, such as smooth ($C^\infty$) functions, presents significant analytical challenges.
In the context of the partition function of random matrices with smooth potentials, Ercolani and McLaughlin \cite{EM03_MR1953782} used Riemann-Hilbert techniques to establish asymptotics of the partition function. 
Recently, the study of the asymptotic expansions for smooth test functions has been significantly developed in the context of strong convergence of random matrices, see for example the works \cite{Schultz05_MR2117954, HT12_MR2922846, Parraud23_MR4567374, parraud2023_unitary, vH_newapproach1, VH_newapproach2}. 
In particular, Parraud established the asymptotic expansion for statistics of smooth functions, first for the GUE \cite{Parraud23_MR4567374} and subsequently for Haar-distributed unitary matrices \cite{parraud2023_unitary}, which successfully bridged the gap between algebraic combinatorial results and functional calculus, proving that quantities like $\E[\tr f(P(\bm{X}^N))]$ admit an asymptotic expansion for smooth $f$.
Subsequently, Chen, Garza-Vargas, Tropp, and van Handel introduced a new approach to asymptotic expansions and strong convergence based on polynomial inequalities and optimal interpolation bounds for $1/N$-samples in \cite{vH_newapproach1, VH_newapproach2}. While their work focused primarily on strong convergence for polynomial norms with matrix coefficients of large dimension, the underlying techniques are powerful and yield asymptotic expansion results.

While these works focused on the \textit{single-trace} case, we extend the analysis to the \textit{multi-trace} case, dealing with expectations of products of traces of the form:
\begin{align*}
\E\left[\tr h_1(X_1^N) \cdots \tr h_r(X_r^N)\right],
\end{align*}
where each $X_i^N = P_i(\text{independent random matrices})$ is a self-adjoint polynomial and the $h_i$ are smooth test functions. 
Such multi-trace observables are indispensable for analyzing the fluctuation properties of the ensemble, including the covariance structure and higher-order cumulants, which characterize the asymptotic independence of the spectral statistics.

In this paper, we focus on the cases of independent standard Gaussian Ensembles and Haar-distributed matrices.
Our analysis relies heavily on the powerful framework developed in \cite{vH_newapproach1,VH_newapproach2}, which is mainly based on the sharp Bernstein-type inequalities and interpolation methods based on ``soft'' arguments.
The extension to multi-trace settings requires careful tracking of the multilinear structure and the interaction between multiple independent random matrices and test functions.

By adapting these techniques to our setting, we develop a unified approach that:
\begin{enumerate}
    \item handles \textit{multi-trace} observables naturally via generalized polynomial approximations.
    \item treats both the \textit{unbounded} (Gaussian Ensembles) and \textit{bounded} (Haar-distributed matrices) cases within a consistent framework, dealing with the rational structure of the unitary group via the rational Bernstein inequalities established in the recent literature.
\end{enumerate}

As applications of our main theorems, we derive three key results:
\begin{enumerate}
    \item We prove the higher-order asymptotic vanishing of cumulants for smooth test functions, extending the universality of second-order freeness.
    \item We establish a Central Limit Theorem (CLT) for the fluctuations of these smooth linear statistics.
    \item We apply our results to matrix integrals with smooth potentials, establishing the existence of the formal asymptotic expansion for the free energy and observables in such models.
\end{enumerate}

\subsection*{Organization of the paper}
The paper is organized as follows.
In Section \ref{sec:prelim}, we introduce the necessary notations, the definitions, and recall the key tools from \cite{VH_newapproach2}, including the Bernstein inequalities.
Section \ref{sec:GUE} is dedicated to the GUE case; we establish a priori bounds and prove the asymptotic expansion for smooth multi-trace statistics.
In Section \ref{sec:unitary}, we treat the Haar unitary case, introducing the rational expression for expectations and applying the rational Bernstein inequality.
Finally, in Section \ref{sec:Application}, we present applications of our main theorems, including the asymptotic vanishing of cumulants, the Central Limit Theorem, and the formal expansion of matrix integrals with smooth potentials.

\subsection*{Acknowledgements}
The authors were supported by JSPS Grant-in-Aid Scientific Research (A) no.
25H00593, and Challenging Research (Exploratory) no. 20K20882 and 23K17299.
They are grateful to Ramon van Handel for useful comments on a preliminary version of this manuscript. 

\section{Preliminaries}\label{sec:prelim}

\subsection{Notations and algebraic preliminaries}
In this section, we setup notations, and recall useful facts, in particular about incidence algebras, as we will need them for cumulants and matrix integrals.  

Throughout the paper, we use the notations
$\mathbb{N} := \{1,2,3,\ldots\}, \mathbb{Z}_{\ge 0} := \{0,1,2,\ldots\}$, 
and for $N\in\mathbb{N}$, we write $[N] := \{1,\ldots,N\}$.
The symmetric group on $N$ letters is denoted by $S_N$, and the free group on $d$ generators by $F_d$.
We denote by $\mathrm{M}_N(\mathbb{C})$ the space of $N\times N$ complex matrices.  
For a matrix $M$, we write $\Tr M$ for its trace, $\tr M := \frac{1}{N}\Tr M$ for its normalized trace, and $\|M\|$ for its operator norm.    
Unless otherwise stated, $C,c>0$ denote universal constants that may vary from line to line.  

We recall the definition of set partitions, and refer to \cite{Zvonkin97_MR1492512}, \cite{Rota64_MR174487}, \cite{Stanley97_MR1442260}.
If $Z$ is a finite set, a partition of $Z$ is a set $\pi=\left\{B_{1}, \ldots, B_{r}\right\}$ of pairwise disjoint non-empty subsets $B_{j}$ of $Z$ whose union is all of $Z$. These sets $B_{j}$ are called the blocks of the partition, and the number of blocks of $\pi$ is simply denoted $\# \pi$. 
The set of partitions of the set of $k$ integers $[k] =\{1, \dots, k\}$, denoted by $\Part(k)$, forms a lattice with respect to the partial order, the meet operation $\wedge$ and the join operation $\vee$ are defined as follows.  
\begin{itemize}
    \item partial order $\leq $: \ For $\pi_1, \pi_2 \in \Part(k)$, we call $\pi_1 \leq \pi_2 $ if and only if each block of $\pi_1$ is contained in some block of $\pi_2$. The maximum partition is denoted by $1_k = \{\{1,\dots , k\}\}$, and the minimum partition is denoted by $0_k=\{\{1\}, \dots , \{k\}\}$ .
    \item meet operation $\wedge$: \ For $\pi_1, \pi_2 \in \Part(k)$, $\pi_1 \wedge \pi_2$ is the greatest lower bound of $\pi_1$ and $\pi_2$.
    \item join operation $\vee$: For $\pi_1, \pi_2 \in \Part(k)$, $\pi_1 \vee \pi_2$ is the least upper bound of $\pi_1$ and $\pi_2$.
\end{itemize}

We define incidence algebra on the lattice $\Part(k)$,
\[I(\Part(k))=\{f:\Part(k) \times \Part(k) \to \mathbb{C}\ | \ f(\pi_1,\pi_2)=0 \mbox{ if } \pi_{1} \nleq \pi_{2}\}\]
with an associative convolution *, for any $\pi_1, \pi_2 \in \Part(k)$,
\begin{equation*}
f * g\left(\pi_{1}, \pi_{2}\right):=\sum_{\pi_{3}: \pi_{1} \leq \pi_{3} \leq \pi_{2}} f\left(\pi_{1}, \pi_{3}\right) g\left(\pi_{3}, \pi_{2}\right). 
\end{equation*}
$I(\Part(k))$ has the unit $\delta$, Kronecker delta, 
\[
\delta(\pi_1, \pi_2) 
= \begin{cases}
    1 & \mbox{if } \pi_1 = \pi_2,\\
    0 & \mbox{otherwise}.
\end{cases}
\]

We define the Zeta function $\zeta$ by 
\begin{equation*}
\zeta\left(\pi_{1}, \pi_{2}\right)
= \begin{cases}
    1 & \mbox{if }\pi_{1} \leq \pi_{2},\\
    0 & \mbox{otherwise},
\end{cases}
\end{equation*}
for $\pi_1, \pi_2 \in \Part(k)$.

We also introduce the M\"obius function, $\Moeb$. We first define \[ \Moeb(0_k, 1_k)=(-1)^{k-1} (k-1)!. \] 
For any $\pi_{1}, \pi_{2} \in \Part(k)$, when we suppose that $\pi_2 = \{B_1, \dots , B_{\# \pi_2}\}$ and that each block $B_i$ of $\pi_2$ is partitioned into $\lambda_i$ blocks in $\pi_1$, the interval $\left[\pi_{1}, \pi_{2}\right]$ in the lattice $\Part(k)$ is isomorphic to
$\left[0_{\lambda_1}, 1_{\lambda_1}\right] \times \cdots \times [0_{\lambda_{\# \pi_2}}, 1_{\lambda_{\# \pi_2}}]$. Now we define  
$$
\Moeb\left(\pi_{1}, \pi_{2}\right)
=\prod_{i=1}^{\# \pi_2} \Moeb(0_{\lambda_i}, 1_{\lambda_i})
=\prod_{i=1}^{\# \pi_2}(-1)^{\lambda_i-1}(\lambda_i-1)!.
$$

We need the fact 
that the convolution product of the M\"obius function $\Moeb$ and the Zeta function $\zeta$ is equal to the Kronecker delta $\delta$,
\begin{equation}\label{MobiusInversion}
\Moeb * \zeta=\delta.
\end{equation}
We call this {\em M\"obius inversion formula}. 
(See \cite{Stanley97_MR1442260} Section 3.7)

Finally, we use the multi-index notation $\mathbf{i}=(i_1,\dots,i_k), \ \mathbf{j}=(j_1,\dots,j_k)$.
For a multi-index $\mathbf{i}= (i_1,...,i_k) \in \{1, \dots , N\}^k$, we view $\mathbf{i}$ as a level partition $\Pi_\mathbf{i} \in \Part(k)$ so that $1 \leq l,m \leq k$ belongs to the same block of $\Pi_\mathbf{i}$ if and only if $i_l=i_m$.

\subsection{Definitions and asymptotic freeness of random matrices}

We begin by defining the classical random matrix ensembles.
\begin{definition}
 The Gaussian Orthogonal Ensemble (GOE) and the Gaussian Unitary Ensemble (GUE) as the probability measures on the space of $N \times N$ real symmetric and complex Hermitian matrices, respectively, with densities with respect to the Lebesgue measure proportional to:
\[
\exp\left( -\frac{N}{4} \operatorname{Tr}(H^2) \right) \quad (\text{GOE})
\quad \text{and} \quad
\exp\left( -\frac{N}{2} \operatorname{Tr}(H^2) \right) \quad (\text{GUE}).
\]   
\end{definition} 

For the Gaussian Symplectic Ensemble (GSE), we avoid working directly with quaternions by representing them as $2 \times 2$ complex matrices. We identify the skew-field $\mathbb{H}$ with the matrix subalgebra spanned by the basis:
\[
\mathbf{1}=\begin{pmatrix}
1 & 0 \\
0 & 1
\end{pmatrix}, \quad 
\mathbf{i}=\begin{pmatrix}
i & 0 \\
0 & -i
\end{pmatrix}, \quad 
\mathbf{j}=\begin{pmatrix}
0 & 1 \\
-1 & 0
\end{pmatrix}, \quad 
\mathbf{k}=\begin{pmatrix}
0 & i \\
i & 0
\end{pmatrix}.
\]
Under this identification, the space of $N \times N$ quaternion Hermitian matrices maps to the space of $2N \times 2N$ complex Hermitian matrices $H$ composed of $N^2$ blocks of size $2\times 2$, where each block is a linear combination of $\{\mathbf{1}, \mathbf{i}, \mathbf{j}, \mathbf{k}\}$. We denote this space of self-dual matrices by $\mathcal{H}_{2N}^{\mathrm{sp}}$.

\begin{definition}
The \emph{GSE} is defined as the probability measure on $\mathcal{H}_{2N}^{\mathrm{sp}}$ with density with respect to the Lebesgue measure proportional to:
\[
\exp\left( -\frac{N}{2} \operatorname{Tr}(H^2) \right),
\]
where the trace is taken over the $2N$-dimensional representation.    
\end{definition}
These normalizations are chosen such that the limiting spectral measure is the Wigner semicircle law supported on $[-2, 2]$.
A fundamental property of these Gaussian ensembles is 
the invariance of their distribution under conjugation by their respective symmetry groups.

We also introduce the Haar-distributed random matrices, which are defined on the classical compact groups $\U(N)$, $\Ort(N)$, and $\Sp(N)$ (the groups of unitary, orthogonal, and symplectic matrices, respectively). It is well known that each of these groups admits a unique left- and right-invariant probability measure, known as the normalized Haar measure.

Next, we define the operator-algebraic counterparts that appear in the large-$N$ limit of our random matrix models.
In this paper, we adopt the framework of a $C^*$-probability space $(\mathcal{A},\tau)$, which consists of a unital $C^*$-algebra $\mathcal{A}$ and a faithful trace $\tau$.

\begin{definition}
Let $(\mathcal{A},\tau)$ be a $C^*$-probability space.
\begin{itemize}
    \item A family of elements $s_1,\ldots,s_d \in \mathcal{A}$ is called a \emph{free semicircular family} if they are freely independent and each $s_i$ follows the standard semicircle law supported on $[-2,2]$.
    \item A family of elements $u_1,\ldots,u_d \in \mathcal{A}$ is called a family of \emph{free Haar unitaries} if they are freely independent and each $u_i$ is a unitary operator whose spectral distribution is the uniform measure on the unit circle.
\end{itemize}
\end{definition}

For a comprehensive treatment of free independence, see \cite{NS06_MR2266879}, \cite{MR3585560_MingoSpeicher17}. 
We also recall that the fundamental connection between these abstract variables and classical random matrices was established by Voiculescu \cite{Voiculescu91_MR1094052}, who proved the asymptotic freeness of independent classical random matrices.
We now restate a simple result of the asymptotic freeness for the case of GUEs and Haar unitary matrices (or more generally, constant matrices rotated by independent Haar unitaries).

\begin{theorem}[Asymptotic freeness for GUE and Haar unitaries]
\label{thm:asymptotic_freeness}
Let $\bm{G}^N=(G_1^N,\ldots,G_d^N)$ be independent $N\times N$ GUE matrices, and let $\bm{U}^N=(U_1^N,\ldots,U_d^N)$ be independent Haar-distributed random matrices in $\U (N)$.
Let $\bm{s}=(s_1,\ldots,s_d)$ and $\bm{u}=(u_1,\ldots,u_d)$ be a free semicircular family and a family of free Haar unitaries in $(\mathcal{A},\tau)$, respectively.

Then, for any non-commutative polynomial $P\in \C \langle x_1,...,x_d \rangle$, the normalized traces converge in expectation as follows:
\begin{align}
\lim_{N\to\infty} \E\big[\tr P(\bm{G}^N)\big] 
= (\tr\otimes\tau)\big[P(\bm{s})\big],
\end{align}
and for any non-commutative polynomial $P\in \C \langle x_1,...,x_d, x_1^*,...,x_d^* \rangle$, we also have
\begin{align}
\lim_{N\to\infty} \E\big[\tr P(\bm{U}^N, \bm{U}^{N*})\big] 
= (\tr\otimes\tau)\big[P(\bm{u}, \bm{u}^*)\big].
\end{align}
\end{theorem}

The asymptotic freeness result can be extended to general gaussian random matrices, and also to the other Haar-distributed random matrices using Weingarten calculus (see \cite{Collins2003, CS06_MR2217291}), and for a comprehensive treatment of these asymptotic freeness, see \cite{NS06_MR2266879}. 

\subsection{Technical preliminary results.}

We begin by introducing some notation and functional-analytic conventions that will be used repeatedly.

Let $\mathcal{P}_q$ be the space of real polynomials $h:\R \to\R$ of degree at most $q$, and let $\mathcal{P}$ denote the space of all real polynomials.  
For a univariate function $h$, we write $h^{(m)}$ for its $m$-th derivative.  
Let $K$ be a compact subset of $\R^r$. 
We denote by $C^m(K)$ the space of functions $\phi: K \to \R$ which admit an extension to a $C^m$-function on some open neighborhood of $K$ in $\R^r$. 
For $\phi \in C^m(K)$, we define the norms:
\begin{equation}\label{def-Cm-norm}
    \|\phi\|_{C^m(K)} := \sum_{0 \le |\bm{\alpha}|\le m} \sup_{\bm{x}\in K} |\partial^{\bm{\alpha}} \phi(\bm{x})|,
\qquad 
\|\phi\|_{K} := \|\phi\|_{C^0(K)} = \sup_{\bm{x}\in K} |\phi(\bm{x})|.
\end{equation}
Here, $\bm{\alpha} = (\alpha_1, \dots, \alpha_r) \in \Z_{\geq 0}^r$ is a multi-index, $|\bm{\alpha}| = \sum_{j=1}^r \alpha_j$, and $\partial^{\bm{\alpha}} = \partial_{x_1}^{\alpha_1} \dots \partial_{x_r}^{\alpha_r}$.

For later use, we recall a few results, representing some key ideas from  
 \cite{VH_newapproach2}.

\begin{lemma}[Bernstein inequality, Lemma 2.1 in \cite{VH_newapproach2}]\label{lemma:BernsteinIneq}
For any $h\in \mathcal{P}_q$ and $\delta>0$, we have
\begin{align*}
	|h^{(m)}(x)| \leq 
	\left( \frac{2q}{\delta \sqrt{1-(x/\delta)^2}} \right)^m 
	\|h\|_{[-\delta, \delta]}
	\quad\text{for all }x\in (-\delta, \delta).
\end{align*}
\end{lemma}

\begin{lemma}[Interpolation from $\tfrac{1}{N}$ samples, Proposition 3.1 in \cite{VH_newapproach2}]\label{lemma:Interpolation}
We have
\begin{align*}
    \|h\|_{[0, \delta]} \le C
	\sup_{\frac{1}{N}\le 2\delta} |h(\tfrac{1}{N})|
\end{align*}    
for every $q\in\mathbb{N}$, $h\in \mathcal{P}_q$, and $0 \le \delta \le 
\frac{1}{24q}$, where $C$ is a universal constant.
\end{lemma}

In contrast to the above lemma which focuses on polynomials, the following lemma allows us to control rational functions, which will be useful when one has to resort to Weingarten calculus. 

\begin{lemma}[Rational Bernstein inequality, Lemma 7.5 in \cite{VH_newapproach2}]\label{lemma:RationalBernstein}
Let $p, q \in \mathbb{N}$ with $p \geq q$, let $f \in \mathcal{P}_{p}$, and define the rational function $r := \frac{f}{g_q}$ where $g_q$ is as defined in \eqref{eq:Def_g}. Then
\begin{align}
\frac{1}{m!} \| r^{(m)} \|_{[-\frac{1}{cp},\frac{1}{cp}]} \leq
\left(
\e^{-p}(Cp)^m + \frac{(Cp)^{2m}}{m!}
\right)
\| r \|_{I_q},
\end{align}
for all $ m \geq 1 $, where $c, C$ are universal constants and $I_q := \left\{ \frac{1}{N} : N \in \mathbb{Z}, |N| > q \right\} $.
\end{lemma}

Many results of this paper rely on the continuity of $k$-linear form and their approximation on dense subspaces (typically, vector spaces of polynomials). The following elementary lemma -- classical for linear maps -- clarifies techniques that are repeatedly used with various Banach spaces. 

\begin{lemma}[Extension of Bounded Multi-linear Maps]\label{lemma:extensionBoundedMultilinear}
Let $H$ be a Banach space with norm $\|\cdot\|$. Let $V$ be a dense vector subspace of $H$. Let $\nu: V^k \to \mathbb{C}$ be a $k$-linear map such that for any $x_1, \ldots, x_k \in V$,
\[
|\nu(x_1, \ldots, x_k)| \le \prod_{i=1}^k \|x_i\|.
\]
Then, $\nu$ extends uniquely to a continuous $k$-linear map $\tilde{\nu}: H^k \to \mathbb{C}$ satisfying the same bound:
\[
|\tilde{\nu}(x_1, \ldots, x_k)| \le \prod_{i=1}^k \|x_i\| \quad \text{for all } x_1, \ldots, x_k \in H.
\]
In particular, if $\nu: V^k \to \mathbb{C}$ is a $k$-linear map such that for any $x_1, \ldots, x_k \in V$,
$\nu(x_1, \ldots, x_k)=0$, $\nu=0$ is its only bounded $k$-linear extension to $H^k$
\end{lemma}

\begin{proof}
Since $V$ is dense in $H$, the product space $V^k$ is dense in $H^k$ with respect to the product topology. Because $H^k$ is a metric space and $\mathbb{C}$ is Hausdorff, any two continuous functions that agree on a dense subset must be identical on the whole space. Thus, if a continuous extension exists, it is unique.

Let $x_1, \ldots, x_k$ be arbitrary vectors in $H$. Since $V$ is dense in $H$, for each index $i \in \{1, \ldots, k\}$, there exists a sequence $(x_i^{(n)})_n \subset V$ such that $\lim_{n \to \infty} x_i^{(n)} = x_i$ in norm.

We define the sequence of scalars $S_n = \nu(x_1^{(n)}, \ldots, x_k^{(n)})$. We claim that $(S_n)$ is a Cauchy sequence in $\mathbb{C}$. Consider the difference:
\[
|\nu(x_1^{(n)}, \ldots, x_k^{(n)}) - \nu(x_1^{(m)}, \ldots, x_k^{(m)})|
\]
Using a telescoping sum argument, we can express the difference as:
\[
\nu(x_1^{(n)}, \ldots) - \nu(x_1^{(m)}, \ldots) = \sum_{j=1}^k \nu\left(x_1^{(m)}, \ldots, x_{j-1}^{(m)}, x_j^{(n)} - x_j^{(m)}, x_{j+1}^{(n)}, \ldots, x_k^{(n)}\right).
\]
Applying the boundedness of $\nu$ on $V$:
\[
|S_n - S_m| \le \sum_{j=1}^k \left( \|x_1^{(m)}\| \cdots \|x_{j-1}^{(m)}\| \cdot \|x_j^{(n)} - x_j^{(m)}\| \cdot \|x_{j+1}^{(n)}\| \cdots \|x_k^{(n)}\| \right).
\]
Since convergent sequences are bounded, there exists $M > 0$ such that $\|x_i^{(n)}\| \le M$ for all $i$ and $n$. Thus:
\[
|S_n - S_m| \le M^{k-1} \sum_{j=1}^k \|x_j^{(n)} - x_j^{(m)}\|.
\]
Since each sequence $(x_i^{(n)})_n$ converges, it is Cauchy in $H$, so $\|x_j^{(n)} - x_j^{(m)}\| \to 0$ as $n, m \to \infty$. Therefore, $(S_n)$ is a Cauchy sequence in $\mathbb{C}$. Since $\mathbb{C}$ is complete, the limit exists. We define $\tilde{\nu}$ by:
\[
\tilde{\nu}(x_1, \ldots, x_k) = \lim_{n \to \infty} \nu(x_1^{(n)}, \ldots, x_k^{(n)}).
\]

Standard arguments show that this limit is independent of the choice of approximating sequences. The $k$-linearity of $\tilde{\nu}$ follows from the $k$-linearity of $\nu$ and the linearity of the limit operation.

Since the norm is a continuous function:
\[
|\tilde{\nu}(x_1, \ldots, x_k)| = \lim_{n \to \infty} |\nu(x_1^{(n)}, \ldots, x_k^{(n)})| \le \lim_{n \to \infty} \prod_{i=1}^k \|x_i^{(n)}\| = \prod_{i=1}^k \|x_i\|.
\]
This confirms that the extension satisfies the same norm bound.
\end{proof}

\section{Asymptotic expansion for multi-trace Gaussian ensembles}
\subsection{The multi-trace GUE case}\label{sec:GUE}

This section is devoted to establishing 
an asymptotic expansion of smooth multi-trace statistics of polynomials of GUE matrices. The results of this expansion will be applied in section \ref{sec:Application} to prove several application results in the GUE case.

Throughout this section, we fix the following setting.
Let $\bm{G}^N = (G_1^N, \ldots, G_d^N)$ be independent GUE matrices of dimension $N$, and let $\bm{s} = (s_1, \ldots, s_d)$ be a free semicircular family. 
Let $r\in \N$ be fixed. 
We also fix self-adjoint non-commutative polynomials $P_i \in 
\mathbb{C}\langle x_1, \ldots, x_d\rangle$ of degree $\tilde{q}_i$ for each $i=1,...,r$. 
We will use the following notations: 
\[
X_i^N := P_i(\bm{G}^N), \qquad X_i^F := P_i(\bm{s}) \qquad i=1,...,r.
\]

We start with the following estimate, which will be used to truncate the Gaussian distribution. This step is a direct modification of Lemma 4.3 ("A priori bounds")
in \cite{VH_newapproach2} to fit our multilinear setting. 

\begin{lemma}
\label{lemma:APrioriBound}
    There exist universal constants $C,c>0$ 
    such that for any $i=1,..., r$,
    \begin{align}\label{ineq:prob_tail}
    \P \left[\|X^N_i\|>K_i\right]\leq C d \e^{-cN},
    \end{align}
    where $K_i \coloneq (Cd)^{\tilde{q}_i}\|X^F_i\|$. 
    \begin{align}\label{ineq:multitrace-bound}
    \left|\E \left[\tr h_1 (X_1^N) \cdots \tr h_r(X_r^N)\right]\right| \leq 2 \|h_1\|_{[-K_1, K_1]} \cdots \|h_r\|_{[-K_r, K_r]}
    \end{align}
    and
    \begin{align}\label{ineq:tr_tail}
    &\left|\E \left[\tr h_1 (X_1^N) \cdots \tr h_r(X_r^N)\cdot 1_{\{\|X_1^N\|>K_1 \mbox{ or }\cdots \mbox{ or }\|X_r^N\|>K_r\}}\right]\right| \notag \\
    & \quad \leq C \sqrt{dr} \cdot \e^{-cN} \|h_1\|_{[-K_1, K_1]} \cdots \|h_r\|_{[-K_r, K_r]},
    \end{align}
    for any $h_1\in \mathcal{P}_{q_1}, \dots , h_r \in \mathcal{P}_{q_r}$, with $q_1,...,q_r$ which satisfy $\tilde{q}_1 q_1 + \cdots + \tilde{q}_r q_r \leq N$.
\end{lemma}

\begin{remark}
In the statement above, the constant $C, c>0$ may depend on the fixed polynomials $P_i$. However, since $P_i$'s are fixed throughout this discussion, we will suppress such dependence and refer to a universal constant $C$ simply as a universal constant. The same convention will be adopted in the subsequent parts of the paper.
\end{remark}

\begin{proof}
    The first inequality \eqref{ineq:prob_tail} follows directly from Lemma 4.3 in \cite{VH_newapproach2}.    
    The remaining inequalities 
    also follow from the same lemma, along with
    Lemma 2.2 in \cite{VH_newapproach2} 
    through an  
    application of H\"older's inequality.
    We start by bounding the expectation of the product of traces as follows:
    \begin{align*}
    |\E [\tr h_1 (X_1^N) \cdots \tr h_r(X_r^N)]| 
    \leq \E [ \sup_{|x_1|\leq \|X_1^N\|} |h_1(x_1)| \cdots \sup_{|x_r|\leq \|X_r^N\|} |h_r(x_r)|].
    \end{align*}
    If $q_i, \ (i=1,...,r)$ satisfies $\tilde{q}_1 q_1 +\cdots + \tilde{q}_r q_r \leq 2N$, then, $t_i \coloneq \frac{2N}{\tilde{q}_i q_i} >0 \ (i=1,...,r)$ satisfies
    $\sum_{i=1}^r \frac{1}{t_i} \leq 1.$
    So applying H\"older's inequality yields
    \begin{align*}
    &\E [ \sup_{|x_1|\leq \|X_1^N\|} |h_1(x_1)| \cdots \sup_{|x_r|\leq \|X_r^N\|} |h_r(x_r)|]\notag \\
    &\leq \E [ \sup_{|x_1|\leq \|X_1^N\|} |h_1(x_1)|^{t_1}]^{\frac{1}{t_1}} \cdots \E[\sup_{|x_r|\leq \|X_r^N\|} |h_r(x_r)|^{t_r}]^{\frac{1}{t_r}}.
    \end{align*}
    Here, by the proof of Lemma 4.3 in \cite{VH_newapproach2} (noting that $q_i \leq \frac{2N}{\tilde{q_i}t_i}$ and combined with Lemma 2.2 in \cite{VH_newapproach2}), we have 
    for each $i=1,...,r$,
    \begin{align*}
    \E [ \sup_{|x_i|\leq \|X_i^N\|} |h_i(x_i)|^{t_i}]
    \leq 2 \|h_i\|^{t_i}_{[-K_i,K_i]}.
    \end{align*}
    Therefore,
    \begin{align*}
    &|\E [\tr h_1 (X_1^N) \cdots \tr h_r(X_r^N)]| \notag \\
    &\quad \leq
    \prod_{i=1}^r
    \left(
    2 \|h_i\|^{t_i}_{[-K_i,K_i]}
    \right)^{\frac{1}{t_i}}
    \leq 2 \|h_1\|_{[-K_1,K_1]}\cdots \|h_r\|_{[-K_r,K_r]}.
    \end{align*}

    Finally, we assume that $h_i \in \mathcal{P}_{q_i}$ and $ q_i, \ (i=1,...,r)$ satisfies $\tilde{q}_1 q_1 +\cdots + \tilde{q}_r q_r \leq N$, by the Cauchy-Schwarz inequality, \eqref{ineq:prob_tail}, and \eqref{ineq:multitrace-bound},
    noting that $q_i \leq \frac{N}{\tilde{q}_i t_i}$ we have, 
    \begin{align*}
    &|\E [\tr h_1 (X_1^N) \cdots \tr h_r(X_r^N) 1_{\{\|X_1^N\|>K_1 \ \mathrm{or} \cdots \mathrm{or} \ \|X_r^N\|>K_r\}}]|\notag \\
    &\leq \left|\E \left[\left(\tr h_1 (X_1^N) \cdots \tr h_r(X_r^N)\right)^2\right]\right|^{\frac{1}{2}} \P\left[{\|X_1^N\|>K_1 \ \mathrm{or} \cdots \mathrm{or} \ \|X_r^N\|>K_r}\right]^{\frac{1}{2}}\notag \\
    &\leq \left(2\|h_1\|^2_{[-K_1,K_1]}\cdots \|h_r\|^2_{[-K_r,K_r]}\right)^{\frac{1}{2}} \left(r\cdot Cd\e^{-cN}\right)^{\frac{1}{2}},
    \end{align*}
    which concludes the proof. 
\end{proof}

Throughout the rest of this section, we fix $K_i$ as in Lemma \ref{lemma:APrioriBound} for each $i=1,...,r$.
Next, we establish the asymptotic expansion for polynomial test functions $h_i \in \mathcal{P} \ (i=1,...,r)$.

\begin{proposition}[Asymptotic expansion for polynomial test functions]\label{prop:ExpansionPolyGUE}
     For any $m \in \mathbb{Z}_{\geq 0}$, 
     there exists an $r$-linear functional $\nu_m:\mathcal{P} ^r\to \mathbb{R}$ such that for any $q_i \in \N$ and $h_i \in \mathcal{P}_{q_i}\ (i=1,...,r)$,
    \begin{align}\label{ineq:nu_Bound}
    |\nu_m(h_1,\dots,h_r)| \leq \frac{(C(\tilde{q}_1 q_1+\cdots +\tilde{q}_r q_r))^{2m}}{m!} \|h_1\|_{[-K_1,K_1]}\cdots \|h_r\|_{[-K_r,K_r]},
    \end{align}
    and in addition for $q_1,...,q_r$ that satisfy $N \geq C(\tilde{q}_1 q_1+\cdots +\tilde{q}_r q_r)$, we have
    \begin{align}\label{ineq:ExpansionPolyGUE}
    &\left|\E [\tr h_1 (X_1^N)\cdots \tr h_r(X_r^N)]-\sum_{k=0}^{m-1} \frac{\nu_k(h_1, \dots ,h_r)}{N^k}\right| \notag \\
    &\quad \leq \frac{(C(\tilde{q}_1 q_1+\cdots +\tilde{q}_r q_r))^{2m}}{m!N^m} \|h_1\|_{[-K_1,K_1]} \cdots \|h_r\|_{[-K_r,K_r]}.
    \end{align}
    Here $C$ is a universal constant, 
    and 
    $K_i := (Cd)^{\tilde{q}_i}\|X_i^F\|$ for $i=1,...,r$.

    Moreover, we have $\nu_{m}=0$ for any odd $m$.
\end{proposition}

\begin{remark}
The vanishing of $\nu_m$ for all odd $m$ reflects the algebraic structure associated with the GUE or Haar unitary matrices, due to its complex symmetries. 
Specifically, the $1/N$-expansions for GUEs or Haar unitary matrices presented in Proposition \ref{prop:ExpansionPolyGUE} and in other theorems below reduce to $1/N^2$-expansion, i.e. 
\begin{align*}  
\E[\tr h_1(X_1^N)\cdots \tr h_r(X_r^N)]  
= \sum_{k=0}^{m-1} \frac{\nu_{2k}(h_1,\ldots,h_r)}{N^{2k}}
+ O\left(N^{-2m}\right),  
\end{align*}
where $\nu_{2k+1}=0$ for all $k\geq 0$. Here we note, however, that this cancellation phenomenon is specific to the GUE or Haar unitary matrices and does not generally hold for other random matrix models; in many ensembles, like GOE/GSE/$\Ort(N)$/$\Sp(N)$, odd-order terms survive.  
\end{remark}

\begin{proof}
    For $h_i \in \mathcal{P}_{q_i}\ (i=1,...,r)$, 
    the existence of a polynomial
    $\Phi_{h_1, \dots , h_r}$ of degree at most $\tilde{q}_1 q_1+\cdots +\tilde{q}_r q_r$ so that
    \begin{align*}
        \E [\tr h_1 (X_1^N)\cdots \tr h_r (X_r^N)]
        =\Phi_{h_1, \dots , h_r}(\tfrac{1}{N})
        =\Phi_{h_1, \dots , h_r}(-\tfrac{1}{N}).
    \end{align*}
    can be verified by the genus expansion for GUE, for example, see section 1.7 in \cite{MR3585560_MingoSpeicher17}. 
    Now, define $\nu_m(h_1,\dots , h_r)\coloneq \frac{\Phi^{(m)}_{h_1, \dots , h_r}(0)}{m!}$ for all $m \in \Z_{\geq 0}$, then $\nu_m$ is an $r$-linear functional, and by Taylor's theorem, we have
    \begin{align}\label{ineq:Taylor}
    \left|\Phi_{h_1, \dots , h_r}(\tfrac{1}{N}) - \sum_{k=0}^{m-1} \frac{\nu_k(h_1, \dots , h_r)}{N^k}\right| 
    \leq \frac{\|\Phi^{(m)}_{h_1, \dots , h_r}\|_{[0,\frac{1}{N}]}}{m! N^m}.
    \end{align}
    Here we note that the symmetry $\Phi_{h_1, \dots , h_r}(\tfrac{1}{N})=\Phi_{h_1, \dots , h_r}(-\tfrac{1}{N})$ implies that all the coefficients of odd powers of $x$ in $\Phi_{h_1, \dots , h_r}$ are zero, hence $\nu_m =0$ for any odd $m$.
    
    Consequently, our goal is to estimate
    $\|\Phi^{(m)}_{h_1, \dots , h_r}\|_{[0,\frac{1}{N}]}$ for large enough $N$.
    By Lemma \ref{lemma:APrioriBound}, for any point $x \in \{\frac{1}{N} \ | \ N\in \N \text{ such that } N \geq \tilde{q}_1 q_1 + \cdots + \tilde{q}_r q_r\}$,
    \begin{align*}
    |\Phi_{h_1, ..., h_r} (x)| \leq 2 \|h_1\|_{[-K_1, K_1]} \cdots \|h_r\|_{[-K_r, K_r]}.
    \end{align*}
    Since $\Phi_{h_1,...,h_r}$ is a polynomial of degree at most $\tilde{q}_1 q_1 + \cdots + \tilde{q}_r q_r$, by applying the interpolation method (Lemma \ref{lemma:Interpolation}), 
    with $\delta \coloneq \frac{1}{24(\tilde{q}_1 q_1 + \cdots + \tilde{q}_r q_r)}$ and some universal constant $C$, 
    \begin{align*}
    \|\Phi_{h_1,...,h_r}\|_{[0,\delta]} 
    &\leq C \sup_{\frac{1}{N} \leq 2\delta} \left|\Phi_{h_1,...h_r}(\tfrac{1}{N})\right|\notag\\
    &\leq C \|h_1\|_{[-K_1, K_1]} \cdots \|h_r\|_{[-K_r, K_r]}.
    \end{align*}
    By the symmetry of $\Phi_{h_1,...,h_r}$, the same bound also holds for $\|\Phi_{h_1,...,h_r}\|_{[-\delta,0]}$. 
    We now apply Bernstein's inequality (Lemma \ref{lemma:BernsteinIneq}) to bound 
    \begin{align}\label{ineq:PhiDerivative_Bound}
    &\|\Phi_{h_1,...,h_r}^{(m)}\|_{[-\frac{\delta}{2}, \frac{\delta}{2}]} 
    \leq (\tfrac{96}{\sqrt{3}}(\tilde{q}_1 q_1 + \cdots + \tilde{q}_r q_r))^{2m} \|\Phi_{h_1,...,h_r}\|_{[-\delta,\delta]} \notag \\
    &\quad \leq (C(\tilde{q}_1 q_1 + \cdots + \tilde{q}_r q_r))^{2m} \|h_1\|_{[-K_1, K_1]} \cdots \|h_r\|_{[-K_r, K_r]},
    \end{align}    
    for all $m \in \mathbb{Z}_{\geq 0}$, with some universal constant $C$, and it concludes the proof of \eqref{ineq:nu_Bound}.

    Finally, to conclude the proof of \eqref{ineq:ExpansionPolyGUE}, we combine the bound \eqref{ineq:PhiDerivative_Bound} with \eqref{ineq:Taylor}, which yields that, for $\frac{1}{N}\leq \frac{\delta}{2} \text{ i.e. } N \geq 48 (\tilde{q}_1 q_1 + \cdots + \tilde{q}_r q_r)$, 
    \begin{align*}
    &\left| \E [\tr h_1 (X_1^N)\cdots \tr h_r (X_r^N)] - \sum_{k=0}^{m-1} \frac{\nu_k(h_1, \dots , h_r)}{N^k}\right| \notag \\
    &\quad \leq \frac{(C(\tilde{q}_1 q_1 + \cdots + \tilde{q}_r q_r))^{2m} \|h_1\|_{[-K_1, K_1]} \cdots \|h_r\|_{[-K_r, K_r]}}{m! N^m}.
    \end{align*}
\end{proof}

We now present the result of this section, Theorem \ref{thm:ExpansionSmoothGUE}, which shows that the expansion holds not only for polynomial test functions but also for arbitrary bounded smooth test functions.
In order to state it, we introduce the following notation. 
For $i=1,...,r$, define $K_i=(Cd)^{\tilde{q}_i}\|X_i^F\|$.
Given any $r$-tuple of smooth functions $h_i \in C^\infty (\R)\ (i=1,...,r)$, set $f_i(\theta):=h_i (K_i \cos \theta)$, and define
$F(x_1,\dots, x_r):= f_1^{(1)}(x_1)\cdots f_r^{(1)}(x_r)$.

\begin{theorem}[Smooth Asymptotic expansion for GUE]\label{thm:ExpansionSmoothGUE}
For some universal constants $C,c>0$, the following holds. 
For any $m \in \mathbb{Z}_{\geq 0}$, there exists an $r$-linear functional $\nu_m:C^\infty (\R) ^r\to\mathbb{R}$ 
extending the map of Proposition \ref{prop:ExpansionPolyGUE},
    such that 
    for every $m, N \in \mathbb{N}$ with $m \leq \frac{N}{2}$, we have 
    \begin{align}\label{ineq:ExpansionSmoothGUE}
    &\left|\E [\tr h_1 (X_1^N)\cdots \tr h_r(X_r^N)]-\sum_{k=0}^{m-1} \frac{\nu_k(h_1, \dots ,h_r)}{N^k}\right| \notag \\
    &\leq \frac{C^{2m+r}(\tilde{q}_1+\cdots +\tilde{q}_r)^{2m}}{m!N^m} \|F\|_{C^{2m}([0,2\pi]^r)} \notag \\
    & \quad + Cdr \e^{-cN} (\|h_1\|_{(-\infty, \infty)}\cdots \|h_r\|_{(-\infty, \infty)} + \|F\|_{[0,2\pi]^r} ),
    \end{align}
where $\|F\|_{C^{2m}([0,2\pi]^r)}$ was defined in Equation \eqref{def-Cm-norm}.
    Moreover, we have $\nu_{m}=0$ for any odd $m$.
\end{theorem}

\begin{proof}[Proof of Theorem \ref{thm:ExpansionSmoothGUE}]
We recall that the Chebyshev polynomial of the first kind $T_n$ is the polynomial of degree $n$ defined by $T_n(\cos \theta)=\cos (n\theta)$. 
We express $h_1\in \mathcal{P}_{q_1}, \dots , h_r \in \mathcal{P}_{q_r}$ as 
\begin{align}\label{eq:Chebyshev}
h_1(x_1)= \sum_{j=0}^{q_1} a_j^{(1)} T_j(K_1^{-1} x_1), \dots , 
h_r (x_r)=\sum_{j=0}^{q_r} a_j^{(r)} T_j(K_r^{-1} x_r) ,
\end{align}
for some real coefficients $a_0^{(i)},...,a_{q_i}^{(i)} \ (i=1,...,r)$.

We divide the proof into two main steps.
In Step 1, we show that $\nu_m \ (m \in \Z_{\geq 0})$ is well-defined as an $r$-linear functional on $C^\infty (\R)^r$.
In Step 2, we verify  
$\nu_m$ 
satisfies equation \eqref{ineq:ExpansionSmoothGUE}.

\medskip\noindent\textbf{Step 1.}

Let $\nu_m \ (m \in \Z_{\geq 0})$ be the $r$-linear functionals defined in Proposition \ref{prop:ExpansionPolyGUE}.
Fix the polynomials $h_1\in \mathcal{P}_{q_1}, \dots , h_r \in \mathcal{P}_{q_r}$ and let $a_0^{(i)},...,a_{q_i}^{(i)} \ (i=1,...,r)$ be the Chebyshev coefficients as \eqref{eq:Chebyshev}.

Let us fix $m \in \Z_{\geq 0}$ for the moment. 
By Proposition \ref{prop:ExpansionPolyGUE}, 
since $\|T_j(K^{-1}\cdot)\|_{[-K,K]}=1$ for all $j$, we have
\begin{align}\label{ineq:CptDistBound}
|\nu_m(h_1,\dots,h_r)|
&\leq \sum_{j_1=0}^{q_1}\cdots \sum_{j_r=0}^{q_r} \left|a_{j_1}^{(1)}\cdots a_{j_r}^{(r)}\right| \left|\nu \left(T_{j_1}(K_1^{-1}\cdot), \dots , T_{j_r}(K_r^{-1} \cdot)\right)\right| \notag \\
&\leq C \sum_{j_1=0}^{q_1}\cdots \sum_{j_r=0}^{q_r} \left|a_{j_1}^{(1)}\cdots a_{j_r}^{(r)}\right| (j_1 + \cdots + j_r)^{2m},
\end{align}
with some constant $C$, which may depend on $m$.
Here, we have
\begin{align}\label{eq:ChebyshevCoeff}
&\sum_{j_1=0}^{q_1}\cdots \sum_{j_r=0}^{q_r} \left|a_{j_1}^{(1)}\cdots a_{j_r}^{(r)}\right| (j_1 + \cdots + j_r)^{2m}\notag \\
&=\sum_{\substack{m_1,...,m_r\geq 0 \text{ s.t.}\\m_1+\cdots+m_r =2m}} \frac{(2m)!}{m_1!\cdots m_r!} \left(\sum_{j_1=0}^{q_1}|a_{j_1}^{(1)}|j_1^{m_1}\right)\cdots \left(\sum_{j_r=0}^{q_r}|a_{j_r}^{(r)}|j_r^{m_r}\right).
\end{align}
Here, for 
any $i=1,...,r$, observe that under the coordinate change $x_i = K_i \cos \theta$, the Chebyshev expansion of $h_i$ becomes the Fourier cosine series of $f_i$, namely $f_i(\theta) = \sum_{j=0}^{q_i} a_{j}^{(i)} \cos(j\theta)$. 
Applying the Cauchy-Schwarz inequality and then Parseval's identity applied to the $(m_i+1)$-th derivative, we obtain
\begin{align}\label{ineq:CauchySchwarz-Parseval}
\sum_{j_i=0}^{q_i}|a_{j_i}^{(i)}|j_i^{m_i}
&\leq \left(\sum_{j_i=1}^{q_i} \frac{1}{j_i^2}\right)^{1/2} \left(\sum_{j_i=1}^{q_i} j_i^{2 m_i+2} |a_{j_i}^{(i)}|^2\right)^{1/2}\notag \\
&\leq C\|f_i^{(m_i+1)}\|_{L^2([0,2\pi])}
\leq \bar{C} \|h_i\|_{C^{m_i+1}([-K_i,K_i])},
\end{align}
with some universal constant $C$, and with some other constant $\bar{C}$, which may depend on $m_i$. 
We now combine \eqref{ineq:CptDistBound} with \eqref{eq:ChebyshevCoeff} and \eqref{ineq:CauchySchwarz-Parseval} to get
\begin{align}\label{ineq:nu_m_Bound}
|\nu_m(h_1,\dots,h_r)|
&\leq C \sum_{\substack{m_1,...,m_r\geq 0 \text{ s.t.}\\m_1+\cdots+m_r =2m}} \frac{(2m)!}{m_1!\cdots m_r!}\|h_1\|_{C^{m_1+1}([-K_1,K_1])}\cdots \|h_r\|_{C^{m_r+1}([-K_r,K_r])}\notag \\
& \leq C \|h_1^{(1)}(x_1)\cdots h_r^{(1)}(x_r)\|_{C^{2m}([-K_1,K_1]\times\cdots\times [-K_r,K_r])},
\end{align}
with some constant $C$, which depends on $m,r$.

By the bound \eqref{ineq:nu_m_Bound}, we can now conclude that $\nu_m$ can be uniquely extended on $(h_1,...,h_r) \in C^\infty (\mathbb{R})^r$ and also satisfy the same bound, 
using 
Lemma \ref{lemma:extensionBoundedMultilinear} together with
the fact that 
polynomials are dense in $C^\infty(\mathbb{R})$ with respect to the topology of uniform convergence of derivatives on compact sets, for any $h_1, \dots, h_r \in C^\infty(\mathbb{R})$.

Here, we note that the extended functional $\nu_m$ on $C^\infty(\mathbb{R})^r$ still satisfies $\nu_m =0$ for any odd $m$, by the continuity of the above extension.

\medskip\noindent\textbf{Step 2.}

We now fix bounded functions $h_1,...,h_r \in C^\infty(\R)$, whose Chebyshev expansions on $[-K_i,K_i] \ (i=1,...,r)$ are written as
\begin{align}\label{eq:ChebushevExpansions}
h_1(x_1)= \sum_{j=0}^{\infty} a_j^{(1)} T_j(K_1^{-1} x_1), ... , 
h_r (x_r)=\sum_{j=0}^{\infty} a_j^{(r)} T_j(K_r^{-1} x_r), 
\end{align}
for any $x_i \in [-K_i, K_i] \ (i=1,...,r)$.
Since the above Chebyshev expansion of $h_i$ on $[-K_i,K_i]$ is the Fourier cosine expansion of the function $f_i(\theta)=h_i(K_i \cos \theta)$ on $[0, 2\pi]$, 
the series in \eqref{eq:ChebushevExpansions} converges uniformly on $[-K_i,K_i]$ for each $i=1,...,r$.
 
Thanks to the
multilinearity of both the expectation and the functionals $\mu_k$, and using the induction on the number of factors $r$, we may reduce the problem to the case where every $h_i$ has a vanishing constant term. 
Indeed, if $h_i$ is constant, the term $\tr h_i(X_i^N)$ is deterministic and factors out of the expectation, reducing the problem to the case of $r-1$ test functions. 
Therefore, without loss of generality, we assume $a_0^{(i)}=0$ for all $i=1,...,r$.

Let $A$ be a set defined by
\begin{align*}
    A:=\{(q_1,...,q_r)\in \Z_{\geq 0}^r \ | \ C(\tilde{q}_1 q_1+\cdots +\tilde{q}_r q_r)\leq N \},
\end{align*}
that is, $A$ is the set of parameters for which \eqref{ineq:ExpansionPolyGUE} holds.
We have
\begin{align}\label{ineq:I+II+III}
&\left|\E[\tr h_1 (X_1^N)\cdots \tr h_r (X_r^N)]-\sum_{k=0}^{m-1} \frac{\nu_k(h_1, \dots , h_r)}{N^k}\right| \notag \\
&\leq \left|\sum_{(q_1, ..., q_r) \in A }\Bigg(\E\left[\tr (a_{q_1}^{(1)} T_{q_1}(K_1^{-1} X_1^N))\cdots \tr (a_{q_r}^{(r)} T_{q_r}(K_r^{-1} X_r^N))\right]\right. \notag \\
& \hspace{13em}\left.-\sum_{k=0}^{m-1} \frac{\nu_k\left(a_{q_1}^{(1)} T_{q_1}(K_1^{-1}\cdot ) , \dots , a_{q_r}^{(r)} T_{q_r}(K_r^{-1} \cdot)\right)}{N^k}\Bigg)\right| \notag \\
&\quad +\left|\sum_{k=0}^{m-1} \frac{\nu_k(h_1, \dots , h_r)-{\displaystyle \sum_{(q_1, ..., q_r) \in A}} \nu_k(a_{q_1}^{(1)} T_{q_1}(K_1^{-1}\cdot ) , \dots , a_{q_r}^{(r)} T_{q_r}(K_r^{-1} \cdot))}{N^k} \right|\notag \\
&\quad +\left|\E\left[\tr h_1 (X_1^N)\cdots \tr h_r (X_r^N)- \sum_{(q_1, ..., q_r) \in A} \tr a_{q_1}^{(1)} T_{q_1}(K_1^{-1} X_1^N)\cdots \tr a_{q_r}^{(r)} T_{q_r}(K_r^{-1} X_r^N)\right]\right| \notag \\
&= (\mathrm{I}) + (\mathrm{II}) +(\mathrm{III}),
\end{align}
We now estimate the three terms (I), (II), and (III) on the right-hand side.

\medskip\noindent \underline{First term (I):} 
Applying the asymptotic expansion for polynomial test functions, \eqref{ineq:ExpansionPolyGUE}, we have
\begin{align*}
(\mathrm{I})
&\leq \sum_{(q_1, ..., q_r) \in A} |a_{q_1}^{(1)}|\cdots |a_{q_r}^{(r)}|
\Bigg|\E \left[\tr T_{q_1}(K_1^{-1}X_1^N)\cdots \tr T_{q_r}(K_r^{-1}X_r^N)\right]\notag \\
&\hspace{14em}-\sum_{k=0}^{m-1} \frac{\nu_k (T_{q_1}(K_1^{-1} \cdot ),..., T_{q_r}(K_r^{-1}\cdot))}{N^k}\Bigg| \notag \\
&\leq \sum_{(q_1, ..., q_r) \in A} |a_{q_1}^{(1)}|\cdots |a_{q_r}^{(r)}|
\frac{(C(\tilde{q}_1 q_1 + \cdots + \tilde{q}_r q_r))^{2m}}{m! N^m},
\end{align*}
where we used the fact that $\|T_j(K^{-1}\cdot)\|_{[-K,K]}=1$ for all $j$.

\medskip\noindent \underline{Second term (II):}
Using \eqref{ineq:nu_Bound}, we have
\begin{align*}
(\mathrm{II})
&\leq \sum_{k=0}^{m-1} \frac{1}{N^k} \left|\nu_k(h_1, \dots , h_r)-\sum_{(q_1, ..., q_r) \in A} \nu_k(a_{q_1}^{(1)} T_{q_1}(K_1^{-1}\cdot ) , \dots , a_{q_r}^{(r)} T_{q_r}(K_r^{-1} \cdot))\right| \notag \\
&\leq \sum_{k=0}^{m-1} \frac{1}{k! N^k} \sum_{(q_1, ..., q_r)\notin A } |a_{q_1}^{(1)}| \cdots |a_{q_r}^{(r)}|(C(\tilde{q}_1q_1 + \cdots +\tilde{q}_r q_r))^{2k} \notag \\
& \leq \sum_{k=0}^{m-1} \frac{1}{k! N^{m-k}}\sum_{(q_1, ..., q_r) \notin A } |a_{q_1}^{(1)}| \cdots |a_{q_r}^{(r)}|\frac{(C(\tilde{q}_1q_1 + \cdots +\tilde{q}_r q_r))^{2m}}{N^m},
\end{align*}
where in the last inequality, we applied the fact that $(q_1, ..., q_r) \notin A$, that is $C(\tilde{q}_1 q_1+\cdots +\tilde{q}_r q_r)>N$, implies 
\begin{align*}
(C(\tilde{q}_1q_1 + \cdots +\tilde{q}_r q_r))^{2k}
\leq \frac{(C(\tilde{q}_1q_1 + \cdots +\tilde{q}_r q_r))^{2m}}{N^{2m-2k}}.
\end{align*}

\medskip\noindent \underline{Third term (III):}
By \eqref{eq:ChebushevExpansions}, we have
\begin{align*}
(\mathrm{III}) 
&=\Bigg|\E\left[\tr h_1 (X_1^N)\cdots \tr h_r (X_r^N)- \sum_{(q_1, ..., q_r) \in A} \tr a_{q_1}^{(1)} T_{q_1}(K_1^{-1} X_1^N)\cdots \tr a_{q_r}^{(r)} T_{q_r}(K_r^{-1} X_r^N)\right]\Bigg|\\
&\leq \Bigg|\E \Bigg[\Bigg(\tr h_1 (X_1^N)\cdots \tr h_r (X_r^N)- \sum_{(q_1, ..., q_r) \in A} \tr a_{q_1}^{(1)} T_{q_1}(K_1^{-1} X_1^N)\cdots \tr a_{q_r}^{(r)} T_{q_r}(K_r^{-1} X_r^N)\Bigg)\\
&\hspace{26em}\cdot 1_{\{\|X_1^N\|\leq K_1, \dots , \|X_r^N\|\leq K_r\}}\Bigg]\Bigg|\\
&\quad + \Bigg|\E \Bigg[\tr h_1 (X_1^N)\cdots \tr h_r(X_r^N)\cdot 1_{\{\|X_1^N\|> K_1\ \mathrm{or} \cdots \mathrm{or}\  \|X_r^N\|> K_r\}}\Bigg]\Bigg|\\
& \quad + \Bigg|\E \left[\sum_{(q_1, ..., q_r) \in A} \Big( \tr a_{q_1}^{(1)} T_{q_1}(K_1^{-1} X_1^N)\cdots \tr a_{q_r}^{(r)} T_{q_r}(K_r^{-1} X_r^N)\Big) \cdot 1_{\{\|X_1^N\|> K_1\ \mathrm{or} \cdots \mathrm{or}\  \|X_r^N\|> K_r\}}\right]\Bigg|\\
& \leq \sum_{(q_1, ..., q_r)\notin A} |a_{q_1}^{(1)}| \cdots |a_{q_r}^{(r)}|\\
& \quad + Cdr\e^{-cN}\|h_1\|_{(-\infty,\infty)}\cdots \|h_r\|_{(-\infty,\infty)} \\
& \quad \quad + C \sqrt{dr}\cdot \e^{-cN} \sum_{(q_1, ..., q_r)\in A}|a^{(1)}_{q_1}|\cdots |a^{(r)}_{q_r}|,
\end{align*}
with some universal constants $C,c>0$, where we applied Lemma \ref{lemma:APrioriBound} in the last inequality.

Combining the above estimates in (I)–(III), we have
\begin{align}\label{ineq:I+II+III_2}
&(\mathrm{I})+(\mathrm{II})+(\mathrm{III}) \notag \\
&\leq \sum_{(q_1, ..., q_r)\in A} |a_{q_1}^{(1)}|\cdots |a_{q_r}^{(r)}| \frac{(C(\tilde{q}_1q_1 + \cdots +\tilde{q}_r q_r))^{2m}}{m! N^m} \notag \\
&\quad +\sum_{k=0}^{m-1} \frac{1}{k! N^{m-k}}\sum_{(q_1, ..., q_r)\notin A } |a_{q_1}^{(1)}| \cdots |a_{q_r}^{(r)}|\frac{(C(\tilde{q}_1q_1 + \cdots +\tilde{q}_r q_r))^{2m}}{N^m} \notag \\
&\quad +\sum_{(q_1, ..., q_r)\notin A} |a_{q_1}^{(1)}| \cdots |a_{q_r}^{(r)}| + Cdr\e^{-cN} \Big( \|h_1\|_{(-\infty,\infty)}\cdots \|h_r\|_{(-\infty,\infty)} +  \sum_{(q_1, ..., q_r)\in A}|a^{(1)}_{q_1}|\cdots |a^{(r)}_{q_r}|\Big) \notag \\
&\leq \sum_{k=0}^m \frac{1}{k! N^{m-k}} \sum_{q_1,...,q_r}|a_{q_1}^{(1)}| \cdots |a_{q_r}^{(r)}| \frac{(C(\tilde{q}_1 q_1 + \cdots +\tilde{q}_r q_r))^{2m}}{N^m} \notag\\
& \quad + Cdr\e^{-cN} \Big( \|h_1\|_{(-\infty,\infty)}\cdots \|h_r\|_{(-\infty,\infty)} +  \sum_{q_1, ..., q_r}|a^{(1)}_{q_1}|\cdots |a^{(r)}_{q_r}|\Big)
\end{align}
Here we note that
\begin{align}\label{ineq:FourierCoeff}
&\sum_{q_1,..,q_r}  |a_{q_1}^{(1)}|\cdots |a_{q_r}^{(r)}| (\tilde{q}_1 q_1 + \cdots +\tilde{q}_r q_r)^m \notag \\
&=\sum_{\substack{m_1,...,m_r \ \mathrm{s.t.}\\ m_1+\cdots +m_r=m}} \frac{m!}{m_1!\cdots m_r!} \tilde{q}_1^{m_1}\cdots \tilde{q}_r^{m_r} \sum_{q_1,..,q_r}|a_{q_1}^{(1)}|q_1^{m_1}\cdots  |a_{q_r}^{(r)}|q_r^{m_r} \notag \\
&\leq C^r \sum_{\substack{m_1,...,m_r \ \mathrm{s.t.} \\ m_1+\cdots +m_r=m}} \frac{m!}{m_1!\cdots m_r!} \tilde{q}_1^{m_1} \cdots \tilde{q}_r^{m_r} \|f_1^{(m_1+1)}\|_{[0,2\pi]}\cdots \|f_r^{(m_r+1)}\|_{[0,2\pi]} \notag \\
&\leq C^r (\tilde{q}_1 +\cdots + \tilde{q}_r)^m \|F\|_{C^{m}([0,2\pi]^r)},
\end{align}
with some universal constant $C$, where we applied the following in the second-to-last inequality: 
for any smooth $h_i$ in the form \eqref{eq:ChebushevExpansions}, we have
\begin{align*}
\sum_{q_i=0}^{\infty}|a_{q_i}^{(i)}|q_i^{n}
&\leq C\|f_i^{(n+1)}\|_{L^2([0,2\pi])}
\quad \text{for any $n \in \N$},
\quad i=1,...,r,
\end{align*}
with some universal constant $C$, which is verified by the fact that \eqref{ineq:CauchySchwarz-Parseval} holds independent of the degree of the polynomial $h_i$.\\
Therefore, combining \eqref{ineq:FourierCoeff} with \eqref{ineq:I+II+III} and \eqref{ineq:I+II+III_2},
\begin{align*}
& \left|\E[\tr h_1 (X_1^N)\cdots \tr h_r (X_r^N)]-\sum_{k=0}^{m-1} \frac{\nu_k(h_1, \dots , h_r)}{N^k}\right|\\
&\leq \sum_{k=0}^m \frac{1}{k! N^{m-k}} \sum_{q_1,...,q_r}|a_{q_1}^{(1)}| \cdots |a_{q_r}^{(r)}| \frac{(C(\tilde{q}_1 q_1 + \cdots +\tilde{q}_r q_r))^{2m}}{N^m} \notag\\
& \quad + Cdr\e^{-cN} \Big( \|h_1\|_{(-\infty,\infty)}\cdots \|h_r\|_{(-\infty,\infty)} +  \sum_{q_1, ..., q_r}|a^{(1)}_{q_1}|\cdots |a^{(r)}_{q_r}|\Big) \\
& \leq \Big(\sum_{k=0}^m \frac{1}{k! N^{m-k}} \Big)\frac{C^{2m+r} (\tilde{q}_1+\cdots +\tilde{q}_r)^{2m}}{N^m} \|F\|_{C^{2m}([0,2\pi]^r)}\\
& \quad + Cdr\e^{-cN} \Big( \|h_1\|_{(-\infty,\infty)}\cdots \|h_r\|_{(-\infty,\infty)} +  \|F\|_{[0,2\pi]^r}\Big).
\end{align*}
This estimate, combined with the fact that we have
\begin{align*}
\sum_{k=0}^m \frac{1}{k! N^{m-k}} 
\leq \frac{1}{m!} \sum_{k=0}^m \left(\frac{m}{N}\right)^{m-k} 
\leq \frac{2}{m!},
\end{align*}
for $m \leq \frac{N}{2}$, provides the desired conclusion, which completes the proof.

\end{proof}

\begin{remark}\label{remark:C^M_GUE}
Although in the proof of Theorem~\ref{thm:ExpansionSmoothGUE} we assume that the test functions are $C^\infty$, this assumption can be relaxed. 
More precisely, the functional $\nu_m$ admits a unique extension to $C^{2m+1}(\R)^r$, and the asymptotic expansion \eqref{ineq:ExpansionSmoothGUE} remains valid for test functions $h_1,\dots,h_r \in C^{2m+1}(\R)$.
\end{remark}

\subsection{Asymptotic expansion for multi-trace GOE and GSE}\label{sec:GOE-GSE}
In this section, we briefly discuss the extension of the asymptotic expansion results, Theorem \ref{thm:ExpansionSmoothGUE}, to the cases of GOE and GSE random matrices.
Most part of the arguments in the GUE case can be adapted to these cases with minor modifications, using the fact that the GOE and GSE are treated as a ``dual'', which one can find as the \textit{supersymmetric duality} in \cite{VH_newapproach2}.
So we only state the results here without going into the proofs.

Let $\bm{G}^N = (G_1^N, \ldots, G_d^N)$ and $\bm{H}^N = (H_1^N, \ldots, H_d^N)$ be independent GOE and GSE matrices of dimension $N$, respectively, and let $\bm{s} = (s_1, \ldots, s_d)$ be a free semicircular family. 
Let $r\in \N$ be fixed. 
We will also fix self-adjoint non-commutative polynomials $P_i \in 
\mathbb{C}\langle x_1, \ldots, x_d\rangle$ of degree $\tilde{q}_i$ for each $i=1,...,r$. 
We will use the following notations for simplicity, 
\[
X_i^N := P_i(\bm{G}^N), \qquad Y_i^N := P_i(\bm{H}^N), \qquad X_i^F := P_i(\bm{s}) \qquad i=1,...,r.
\]

In this setting, the supersymmetric duality allows us to observe the following. 
For any polynomial test functions $h_i \in \mathcal{P}_{q_i} \ (i=1,...,r)$, there exists some polynomial $\Phi_{h_1,...,h_r}$ of degree at most $\tilde{q}_1 q_1 + \cdots + \tilde{q}_r q_r$ such that
\begin{align*}
&\E \left[ \tr h_1(X_1^N)\cdots \tr h_r(X_r^N) \right]
= \Phi_{h_1,...,h_r} \left(\tfrac{1}{N}\right), \\
&\E \left[ \tr h_1(Y_1^N)\cdots \tr h_r(Y_r^N) \right]
= \Phi_{h_1,...,h_r} \left(-\tfrac{1}{2N}\right).
\end{align*}
This can be checked by the genus expansions for GOE and GSE in \cite{BP09_MR2480549}.

With the above modification, we can now state the following results, which are analogues of Theorem \ref{thm:ExpansionSmoothGUE}.

\begin{theorem}[Asymptotic expansions for multi-trace GOE and GSE]\label{thm:ExpansionSmoothGOE-GSE}
For any $m \in \Z_{\geq 0}$, there exists an $r$-linear functional $\nu_m$ on $C^\infty(\R)^r$ such that the folllowing holds with some universal constants $C,c>0$:\\
fix any bounded $r$-tuple of smooth functions $h_i \in C^\infty(\R) \ (i=1,...,r)$, and define 
\begin{align*}
    &f_i(\theta):=h_i (K_i \cos \theta)
    \quad \text{with} \quad
    K_i=(Cd)^{\tilde{q}_i}\|X_i^F\| , \quad i=1,...,r,\\
    &\quad \text{and } \quad F(x_1,\dots, x_r):= f_1^{(1)}(x_1)\cdots f_r^{(1)}(x_r);
\end{align*} 
then, for every $m, N \in \mathbb{N}$ with $m \leq \frac{N}{2}$, we have 
\begin{align}\label{ineq:ExpansionSmoothGOE-GSE}
    &\left|\E [\tr h_1 (X_1^N)\cdots \tr h_r(X_r^N)]-\sum_{k=0}^{m-1} \frac{\nu_k(h_1, \dots ,h_r)}{N^k}\right| \vee \left|\E [\tr h_1 (Y_1^N)\cdots \tr h_r(Y_r^N)]-\sum_{k=0}^{m-1} \frac{\nu_k(h_1, \dots ,h_r)}{(-2N)^k}\right| \notag \\
    &\leq \frac{C^{2m+r}(\tilde{q}_1+\cdots +\tilde{q}_r)^{2m}}{m!N^m} \|F\|_{C^{2m}([0,2\pi]^r)} 
    + Cdr \e^{-cN} (\|h_1\|_{(-\infty, \infty)}\cdots \|h_r\|_{(-\infty, \infty)} + \|F\|_{[0,2\pi]^r} ).
\end{align}
\end{theorem}

\section{Asymptotic expansion for Haar distributed random variables on compact matrix groups}
\subsection{The Haar Unitary case}\label{sec:unitary}

In this section, we focus on Haar-distributed unitary matrices.
The overall strategy of the proofs follows that of the GUE case, though several distinctions arise.
Unlike in the GUE setting, Haar unitary matrices are inherently bounded, so truncation is unnecessary.
However, we must now handle rational functions in $\frac{1}{N}$, rather than polynomials.

Throughout this subsection, 
$\bm{U}^N = (U_1^N, \ldots, U_d^N)$ 
is a $d$-tuple of independent Haar distributed unitary matrices of dimension $N$, in $\U (N)$, and 
$\bm{u} = (u_1, \ldots, u_d)$  
are free Haar unitaries. 
Let $r \in \N$ be fixed. 
We also fix self-adjoint non-commutative polynomials $P_i \in 
\mathbb{C}\langle x_1, \ldots, x_d, x_1^*,\ldots, x_d^*\rangle$ of degree $\tilde{q}_i$ for each $i=1,...,r$. 
Following the lines of the GUE setup, we use the following notation
\[
X_i^N := P_i(\bm{U}^N,\bm{U}^{N*}), \qquad X_i^F := P_i(\bm{u},\bm{u}^*) \qquad i=1,...,r.
\]
We introduce the following notation for $i=1,\dots,r$ :
$$K_i :=\|P_i\|_{C^*(F_d)} = \sup_{n \in \mathbb{N}} \sup_{W_1,\dots,W_d \in \U (n)} \|P_i(W_1,\dots,W_d, W_1^*,\dots,W_d^*)\|, $$
representing the norm of $P_i$ in the full C*-algebra of the free group $F_d$ with $d$ free generators. The equality between the operator norm and the supremum over unitary matrices follows from the fact that $C^*(F_d)$ is residually finite dimensional (see \cite{Choi1980_MR590864}). Consequently, we have $\|X_i^N\|\leq K_i$ almost surely for any $N$.

We begin by establishing an analogue of Proposition \ref{prop:ExpansionPolyGUE} for Haar-distributed unitary matrices. 

\begin{proposition}[Asymptotic expansions for polynomial test functions, Haar unitary case]\label{prop:ExpansionPolyUnitary}
For any $m \in \Z_{\geq 0}$, there exists an $r$-linear functional $\mu_m$ on $\mathcal{P}^r$ so that for any $q_i \in \N$ and $h_i \in \mathcal{P}_{q_i} \ (i=1,...,r)$,
\begin{align}\label{ineq:mu_bound}
|\mu_m(h_1,...,h_r)|
\leq \left((C\bar{p})^m + \frac{(C\bar{p})^{2m}}{m!}\right) \|h_1\|_{[-K_1,K_1]}\cdots \|h_r\|_{[-K_r,K_r]},
\end{align}
and in addition for all $N \in \N$,
\begin{align}\label{ineq:ExpansionPolyUnitary}
& \left|
\E \left[ \tr h_1(X_1^N)\cdots \tr h_r(X_r^N) \right]
-\sum_{k=0}^{m-1} \frac{\mu_k(h_1,...,h_r)}{N^k}\right| \notag \\
& \quad 
\leq\frac{1}{N^m}\left((C\bar{p})^m + \frac{(C\bar{p})^{2m}}{m!}\right) \|h_1\|_{[-K_1,K_1]}\cdots \|h_r\|_{[-K_r,K_r]}.
\end{align}
Here $C$ is a universal constant, and we define
\begin{align*}
p:=\tilde{q_1}q_1+\cdots +\tilde{q_r}q_r, \quad 
\bar{p}:=p(1+\log p).
\end{align*}

Moreover, we have $\mu_{m}=0$ for any odd $m$.
\end{proposition}

\begin{remark}
Although our main interest in this manuscript is the study of $\E \left[ \tr h_1(X_1^N)\cdots \tr h_r(X_r^N) \right]$,
at an earlier stage of this work, 
 we were also looking as a preliminary step at an expansion for $\E \left[ \tr h_1(X_1^N)\cdots  h_r(X_r^N) \right]$ as a slightly simpler preliminary step. For this model, there is an obvious variant of Theorem \ref{thm:ExpansionSmoothUnitary} with different $r$-linear functions. 
We thank Ramon van Handel for pointing out an approach that reduces the multi-trace problem to the single-trace setting treated in \cite{VH_newapproach2}.
The key observation is that the multi-trace statistic under consideration in Proposition \ref{prop:ExpansionPolyUnitary} can be written as a single-trace case. Precisely, for any polynomials $h_1,...,h_r$, one may express
\begin{align*}
\E \left[ \tr h_1(X_1^N)\cdots \tr h_r(X_r^N) \right]
=\E \left[\tr (U_{d+1} h_1(X_1^N) U_{d+1}^{*} U_{d+2} h_2(X_2^N) U_{d+2}^{*} \cdots U_{d+r} h_r(X_r^N) U_{d+r}^{*})\right],
\end{align*}
where $U_{d+1},...,U_{d+r}$ are additional independent Haar unitary matrices of dimension $N$, which are also independent of $\bm{U}^N = (U_1^N,...,U_d^N)$.
Either way, adapting the linear results of \cite{VH_newapproach2} to an $r$-linear context remains indispensable in both cases, so we chose to settle for the study of $\E \left[ \tr h_1(X_1^N)\cdots \tr h_r(X_r^N) \right]$.
\end{remark}

Before proving Proposition \ref{prop:ExpansionPolyUnitary}, we need to show that for polynomials $h_i (i=1,...,r)$, the quantity $\E[\tr h_1(X_1^N)\cdots \tr h_r(X_r^N)]$ is a rational function of $\frac{1}{N}$—not a polynomial, as in the GUE case.

\begin{lemma}[Rational Expression]\label{lemma:RationalExpression}
For any $q_i \in \N  \ (i=1,...,r)$, 
define
$p:=\tilde{q_1}q_1+\cdots +\tilde{q}_r q_r$.
Then for $h_i \in \mathcal{P}_{q_i} \ (i=1,...,r)$, there exists a rational function 
\[
\Psi_{h_1,...,h_r}:=\frac{f_{h_1,...,h_r}}{g_{p}},
\]
with polynomials $f_{h_1,...,h_r}, g_p \in \mathcal{P}_{\lfloor 3p(1+\log p)\rfloor}$, so that 
\begin{align}\label{eq:symmetry_unitary}
\E \left[\tr h_1(X_1^N)\cdots \tr h_r(X_r^N)\right]
=\Psi_{h_1,...,h_r} (\tfrac{1}{N})
=\Psi_{h_1,...,h_r}(-\tfrac{1}{N})
\end{align}
for any $N \in N$ satisfying $N>p$, where we define a polynomial $g_q$ for any $q\in \N$ as 
\begin{align}\label{eq:Def_g}
g_q(x):=\prod_{j=1}^q (1-(jx)^2)^{\lfloor \frac{q}{j} \rfloor}.
\end{align}
\end{lemma}

\begin{proof}
This proof is similar to the proof of Lemma 7.1 in \cite{VH_newapproach2}, using Weingarten calculus.
For each $i=1,...,r$, let $w_i(U_1^N,...,U_d^N)$ denote a reduced word in the Haar-distributed unitary matrices $U_1^N,...,U_d^N$, and their adjoints $U_1^{N*},...,U_d^{N*}$.

Let $k$ be the total length of the words $w_1,...,w_r$.
We first assume that none of the words $w_i \ (i=1,...,r)$ is the identity, and that for any $l=1,...,d$, each of $U_l^N$ and $U_l^{N*}$ appears the same number of times, say $k_l$, among the words $w_1,...,w_r$ in total. Note that $k=2(k_1+\cdots+k_d)$.
There exists a permutation $\pi \in S_k$ depending on $w_1,...,w_r$, such that, if we denote the multi-indices by
\begin{align*}
&\bm{i}=(i_n)_{n=1,...,k}^{}
:=\big(i^{(l)}_{m},\bar{i}^{(l)}_m \big)_{\substack{l=1,...,d, \\ m=1,...,k_l}}\in [N]^k,\\
&\bm{j}=(j_n)_{n=1,...,k}^{}
:=\big(j^{(l)}_{m},\bar{j}^{(l)}_m \big)_{\substack{l=1,...,d, \\ m=1,...,k_l}}\in [N]^k,
\end{align*}
we have
\begin{align*}
&\E \left[\tr w_1(U_1^N,...,U_d^N)\cdots \tr w_r(U_1^N,...,U_d^N)\right]\\
&=\sum_{\substack{\bm{i},\bm{j}\in [N]^k\\ \text{ s.t. }i_n=j_{\pi(n)} \ (\forall n=1,...,k) }} \Bigg(\E\left[(U_1^N)_{i^{(1)}_1 j^{(1)}_1}\cdots (U_1^N)_{i^{(1)}_{k_1}j^{(1)}_{k_1}}(U_1^{N*})_{\bar{i}^{(1)}_1 \bar{j}^{(1)}_1}\cdots (U_1^{N*})_{\bar{i}^{(1)}_{k_1}\bar{j}^{(1)}_{k_1}}\right]\\
& \qquad \qquad \cdots \E\left[(U_d^N)_{i^{(d)}_{1} j^{(d)}_{1}}\cdots (U_d^N)_{i^{(d)}_{k_d}j^{(d)}_{k_d}}(U_d^{N*})_{\bar{i}^{(d)}_{1} \bar{j}^{(d)}_{1}}\cdots (U_d^{N*})_{\bar{i}^{(d)}_{k_d}\bar{j}^{(d)}_{k_d}}\right]\Bigg).
\end{align*}
By applying the Weingarten calculus for Haar unitary matrices, for each $l=1,...,d$, we have
\begin{align}\label{eq:TrWord}
&\E\left[(U_l^N)_{i^{(l)}_{1} j^{(l)}_{1}}\cdots (U_l^N)_{i^{(l)}_{k_l}j^{(l)}_{k_l}}(U_l^{N*})_{\bar{i}^{(l)}_{1} \bar{j}^{(l)}_{1}}\cdots (U_l^{N*})_{\bar{i}^{(l)}_{k_l}\bar{j}^{(l)}_{k_l}}\right]\\
&=\sum_{\substack{\sigma, \tau \in S_{k_l} \text{ s.t. }  \\ i^{(l)}_m=\bar{i}^{(l)}_{\sigma(m)},j^{(l)}_m=\bar{j}^{(l)}_{\tau(m)} (\forall m= 1,..,k_l)}} \Wg^{\mathcal{U}}_{k_l} (\sigma^{-1}\tau,N),
\end{align}
using the Weingarten function $\Wg^{\mathcal{U}}$ for unitary matrices.
From the above equations, it follows that
$\E \left[\tr w_1(U_1^N,...,U_d^N)\cdots \tr w_r(U_1^N,...,U_d^N)\right]$ is a sum of functions of type $\Wg^{\mathcal{U}}_{k'} (\sigma,N)$, for $k'\le k$ and $\sigma$ permutations on $k'$ elements.
For any $n\in \N$ and $\sigma \in S_n$, the Weingarten function $\Wg^{\mathcal{U}}_n(\sigma,N)$ satisfies
\begin{align}\label{eq:DefWeingarten}
\Wg^{\mathcal{U}}_{n}(\sigma,N)=\frac{1}{n!}\sum_{\lambda \vdash n} \frac{f^\lambda}{C_\lambda (N)} \chi^\lambda (\sigma),
\end{align}
where the sum runs over all partitions $\lambda$ of $n$. 
Here $\chi^\lambda$ denotes the irreducible character of $S_n$, $f^\lambda:=\chi^\lambda (\operatorname{id}_n)$, and $C_\lambda (N)$ is a polynomial in $N$ given by
\begin{align}
C_\lambda(N):=\prod_{(i,j)\in \lambda} (N+j-i),
\end{align}
We refer to \cite{CS06_MR2217291}, \cite{Matsumoto13_MR3077830} for details.
In particular, the polynomial $C_\lambda(N)$ can also be written as
\begin{align*}
C_\lambda(N)=\prod_{s=-n}^n (N+s)^{\omega_s(\lambda)},
\end{align*}
where 
\begin{align*}
\omega_s(\lambda)&:=\# \{(i,j)\in \lambda \ | \ i-j=s\}\\
&\leq \# \{(i,j)\in [n]^2 \ | \ ij\leq n, i-j=s\}\\
&\leq \# \{j \in [n] \ | \ (j+|s|)j\leq n\}
\leq \left\lfloor \frac{n}{|s|+1}\right\rfloor .
\end{align*}
Together with  
this observation about the Weingarten functions, \eqref{eq:TrWord} 
yields that there exists some polynomial $b_{w_1,...,w_r}$ in $N$, such that
\begin{align*}
\E \left[\tr w_1(U_1^N,...,U_d^N)\cdots \tr w_r(U_1^N,...,U_d^N)\right]
=\frac{b_{w_1,...,w_r}(N)}{N^k \prod_{j=1}^k (N^2-j^2)^{\lfloor \frac{k}{j}\rfloor}}.
\end{align*}
Let $\Sigma$ denote the degree of the denominator on the right-hand side. The following estimate holds: 
\[
\deg b_{w_1,...,w_r} 
\leq \Sigma 
=k+2\sum_{j=1}^k \left\lfloor \frac{k}{j}\right\rfloor  
\leq 3k(1+\log k).
\]
By dividing both the numerator and the denominator by $N^\Sigma$ we find that there exists a polynomial $f_{w_1,...,w_r}$ such that 
\begin{align}\label{eq:WordRational}
\E \left[\tr w_1(U_1^N,...,U_d^N)\cdots \tr w_r(U_1^N,...,U_d^N)\right]
=\frac{f_{w_1,...,w_r} (\frac{1}{N})}{g_k(\frac{1}{N})}
\end{align}
for all $N>k$, with $f_{w_1,...,w_r},g_k \in \mathcal{P}_\Sigma$.

If, for some $1\leq l\leq d$, $U_l^N$ and $U_l^{N*}$ do not appear the same number of times among the words $w_1,...,w_r$ in total, the left-hand side of \eqref{eq:WordRational} vanishes, the equality \eqref{eq:WordRational} still holds. 
Furthermore, if one of the words $w_i$ is the identity, we can reduce the statement to the case with $r-1$ words, and
noting that when all words are the identity, the left-hand side equals 1, so by induction, \eqref{eq:WordRational} therefore holds in general.

Next, note that 
$\E \left[\tr h_1(U_1^N,...,U_d^N)\cdots \tr h_r(U_1^N,...,U_d^N)\right]$ 
can be written as a linear combination of terms of the form 
$\E \left[\tr w_1(U_1^N,...,U_d^N)\cdots \tr w_r(U_1^N,...,U_d^N)\right]$, 
where the total length of the words $w_1,...,w_r$ is at most $k=\tilde{q}_1q_1+\cdots+\tilde{q}_r q_r=:p$.
Thus, $\E \left[\tr h_1(X_1^N)\cdots \tr h_r(X_r^N)\right]
=\Psi_{h_1,...,h_r} (\tfrac{1}{N})$ follows from \eqref{eq:WordRational}.

Moreover, due to the unitary symmetry (cf \cite{CS06_MR2217291}, \cite{Matsumoto13_MR3077830} for details) we have 
\begin{align}
\Psi_{h_1,...,h_r} (\tfrac{1}{N}) = \Psi_{h_1,...,h_r} (-\tfrac{1}{N}),
\end{align}
which shows that the power series expansion of $\Psi_{h_1,...,h_r}$ involves only even powers of $\frac{1}{N}$, and this completes the proof.
\end{proof}

We now turn to the proof of Proposition \ref{prop:ExpansionPolyUnitary}, the asymptotic expansions of $\E\left[\tr h_1(X_1^N)\cdots \tr h_r(X_r^N)\right]$ for polynomial test functions $h_1,...,h_r \in \mathcal{P}$, representing the analogue of Proposition \ref{prop:ExpansionPolyGUE} for Haar-distributed unitary matrices.
As before, the structure of the proof essentially follows the same outline as in the case $r=1$, shown in \cite{VH_newapproach2}.

Using lemma \ref{lemma:RationalBernstein}, we can now prove the asymptotic expansions of $\E\left[\tr h_1(X_1^N)\cdots \tr h_r(X_r^N)\right]$ for polynomial test functions in Proposition \ref{prop:ExpansionPolyUnitary}.
Note that as no truncation is needed in the Haar setting as the random matrix model is already bounded, unlike in the Gaussian
case.

\begin{proof}[Proof of Proposition \ref{prop:ExpansionPolyUnitary}]
Let $\Psi_{h_1,...,h_r}$ be as in Lemma \ref{lemma:RationalExpression}. 
Define $\mu_m(h_1,...,h_r):=\frac{\Psi_{h_1,...,h_r}^{(m)}(0)}{m!}$ for all $m \in \Z_{\geq 0}$, then $\mu_m$ is an $r$-linear functional, and note that $\mu_m=0$ for any odd $m$ due to  the symmetry of \eqref{eq:symmetry_unitary}.
By Taylor's theorem, we have
\begin{align}\label{ineq:Taylor_Psi}
\left|\Psi_{h_1, \dots , h_r}(\tfrac{1}{N}) - \sum_{k=0}^{m-1} \frac{\mu_k(h_1, \dots , h_r)}{N^k}\right| 
\leq \frac{\|\Psi^{(m)}_{h_1, \dots , h_r}\|_{[0,\frac{1}{N}]}}{m! N^m}.
\end{align}
Since $\|X_i^N\|\leq K_i \ (i=1,...,r)$ a.s.,
\begin{align*}
\left|\Psi_{h_1,...,h_r}(\tfrac{1}{N})\right|
=\left|\E \left[\tr h_1(X_1^N)\cdots \tr h_r(X_r^N)\right]\right|
\leq \|h_1\|_{[-K_1,K_1]}\cdots\|h_r\|_{[-K_r,K_r]},
\end{align*}
for any $N\in \N$.
Since Lemma \ref{lemma:RationalExpression} yields the expression
$\Psi_{h_1,...,h_r}\left(\tfrac{1}{N}\right)=\frac{f_{h_1,...,h_r}\left(\tfrac{1}{N}\right)}{g_p\left(\tfrac{1}{N}\right)}$, where $f_{h_1,...,h_r}, g_p \in \mathcal{P}_{\lfloor3\bar{p}\rfloor}$, applying the rational Bernstein inequality (Lemma \ref{lemma:RationalBernstein}), we have for all $m,N\in \N$ that
\begin{align}\label{ineq:PsiDerivative_Bound}
\frac{1}{m!}\|\Psi_{h_1,...,h_r}^{(m)}\|_{[-\frac{1}{c\bar{p}},\frac{1}{c\bar{p}}]}
&\leq \left(\e^{-\bar{p}}(C\bar{p})^m + \frac{(C\bar{p})^{2m}}{m!}\right)\|\Psi_{h_1,...,h_r}\|_{\{\frac{1}{N} \ | \ N \in \Z, |N|>p\}}\notag \\
&\leq \left((C\bar{p})^m + \frac{(C\bar{p})^{2m}}{m!}\right)\|h_1\|_{[-K_1,K_1]}\cdots\|h_r\|_{[-K_r,K_r]},
\end{align}
with some universal constants $C,c>0$. 
This also holds for $m=0$ since we have
\begin{align*}
|\mu_0(h_1,...,h_r)|
\leq \limsup_{N \to \infty} \left|\E \left[\tr h_1(X_1^N)\cdots \tr h_r(X_r^N)\right]\right|
\leq \|h_1\|_{[-K_1,K_1]}\cdots\|h_r\|_{[-K_r,K_r]},
\end{align*} 
which completes the proof of \eqref{ineq:mu_bound}.

Now we combine the bound \eqref{ineq:PsiDerivative_Bound} with \eqref{ineq:Taylor_Psi}, and this yields that, for $\frac{1}{N}\leq \frac{1}{c\bar{p}}$ i.e. $N \geq c\bar{p}$,
\begin{align*}
\left|\Psi_{h_1, \dots , h_r}(\tfrac{1}{N}) - \sum_{k=0}^{m-1} \frac{\mu_k(h_1, \dots , h_r)}{N^k}\right| 
\leq \frac{1}{N^m} \left((C\bar{p})^m + \frac{(C\bar{p})^{2m}}{m!}\right)\|h_1\|_{[-K_1,K_1]}\cdots\|h_r\|_{[-K_r,K_r]},
\end{align*}
which proves \eqref{ineq:ExpansionPolyUnitary} in the case $N \geq c\bar{p}$.
In the case $N < c\bar{p}$, \eqref{ineq:mu_bound} shows that
\begin{align*}
\left|\Psi_{h_1,...,h_r} \left(\tfrac{1}{N}\right)-\sum_{k=0}^{m-1} \frac{\mu_k(h_1,...,h_r)}{N^k}\right|
&\leq \sum_{k=0}^{m-1} \frac{1}{N^k} \left((C\bar{p})^k + \frac{(C\bar{p})^{2k}}{k!}\right)\|h_1\|_{[-K_1,K_1]}\cdots\|h_r\|_{[-K_r,K_r]}\\
&\leq \frac{(C\bar{p})^m}{N^m} \sum_{k=0}^{m-1} \left(1 + \frac{(C\bar{p})^{k}}{k!}\right)\|h_1\|_{[-K_1,K_1]}\cdots\|h_r\|_{[-K_r,K_r]}.
\end{align*}
Now if $C\bar{p}\leq m$, 
\[
\sum_{k=0}^{m-1} \left(1 + \frac{(C\bar{p})^{k}}{k!}\right)
\leq \sum_{k=0}^{m-1} \left(1+\frac{m^k}{k!}\right)
\leq C^m,
\]
and if $C\bar{p}> m$, 
\[
\sum_{k=0}^{m-1} \left(1 + \frac{(C\bar{p})^{k}}{k!}\right)
\leq \sum_{k=0}^{m-1} \left(1 + \frac{(C\bar{p})^{m}}{m!}\right)
=m \left(1 + \frac{(C\bar{p})^{m}}{m!}\right),
\]
so in each case, \eqref{ineq:ExpansionPolyUnitary} holds, which completes the proof.
\end{proof}

We now proceed to prove the main theorem in this section, the asymptotic expansions for arbitrary smooth test functions, representing the analogue of Theorem \ref{thm:ExpansionSmoothGUE} for Haar-distributed unitary matrices.

\begin{theorem}[Smooth Asymptotic expansion, Haar unitary case]\label{thm:ExpansionSmoothUnitary}
For any $m \in \Z_{\geq 0}$, there exists an $r$-linear functional $\mu_m$ on $C^\infty (\R)^r$ such that the following holds with some universal constant $C>0$.:\\
Fix any $r$-tuple of smooth functions $h_i \in C^\infty (\R)\ (i=1,...,r)$, 
and define 
\begin{align*}
&f_i(\theta):=h_i (K_i \cos \theta)
\quad \text{with} \quad
K_i=\|P_i\|_{C^* (F_d)} , \quad i=1,...,r,\\
&\quad \text{and } \quad F(x_1,\dots ,x_r):= f_1^{(1)}(x_1)\cdots f_r^{(1)}(x_r);
\end{align*} 
then, for every $m, N \in \mathbb{N}$, and for any $u \in (0,1)$, we have
\begin{align}\label{ineq:ExpansionSmoothUnitary}
&\left|\E [\tr h_1 (X_1^N)\cdots \tr h_r(X_r^N)]-\sum_{k=0}^{m-1} \frac{\nu_k(h_1, \dots ,h_r)}{N^k}\right| \notag \\
&\leq \frac{C^{m+r}}{N^m} \Bigg( \frac{(\tilde{q}_1+\cdots +\tilde{q}_r)^{\lceil m^\prime \rceil}}{u^m} \|F\|_{C^{\lceil m^\prime \rceil}([0,2\pi]^r)} + \frac{(\tilde{q}_1+\cdots +\tilde{q}_r)^{\lceil 2m^\prime \rceil}}{m! u^{2m}} \|F\|_{C^{\lceil 2m^\prime \rceil}([0,2\pi]^r)} \Bigg) ,
\end{align}
where $m^\prime := (1+u)m$.

Moreover, we have $\mu_{m}=0$ for any odd $m$.
\end{theorem}

\begin{proof}
The proof follows the same outline as that of Theorem \ref{thm:ExpansionSmoothGUE}, using Proposition \ref{prop:ExpansionPolyUnitary} in place of Proposition \ref{prop:ExpansionPolyGUE}.
We first note that for any $m \in \Z_{\geq 0}$, by Proposition \ref{prop:ExpansionPolyUnitary}, there exists an $r$-linear functional $\mu_m$ on $\mathcal{P}^r$ satisfying \eqref{ineq:mu_bound} and \eqref{ineq:ExpansionPolyUnitary}, and it can be extended to $C^\infty (\R)^r$ in the same way as Step 1 in the proof of Theorem \ref{thm:ExpansionSmoothGUE}.
We also have $\mu_m=0$ for any odd $m$ due to the continuity of the extension.

Now we fix any $r$-tuple of functions $h_i \in C^\infty (\R) \ (i=1,...,r)$, where their Chebyshev expansions on $[-K_i,K_i] \ (i=1,...,r)$ are given by
\begin{align}\label{eq:ChebushevExpansions_unitary}
h_1(x_1)= \sum_{j=0}^{\infty} a_j^{(1)} T_j(K_1^{-1} x_1), ... , 
h_r (x_r)=\sum_{j=0}^{\infty} a_j^{(r)} T_j(K_r^{-1} x_r), 
\end{align}
for any $x_i \in [-K_i, K_i] \ (i=1,...,r)$.
Here, we note that because of the same reasoning as in the proof of Theorem \ref{thm:ExpansionSmoothGUE}, we may assume without loss of generality that $a_0^{(i)}=0 \ (i=1,...,r)$ in the rest of the proof.

Then we have for any $m,N \in \N$ that
\begin{align*}
&\left|\E[\tr h_1 (X_1^N)\cdots \tr h_r (X_r^N)]-\sum_{k=0}^{m-1} \frac{\mu_k(h_1, \dots , h_r)}{N^k}\right| \notag \\
&= \left|\sum_{q_1, ..., q_r}\Bigg(\E\left[\tr (a_{q_1}^{(1)} T_{q_1}(K_1^{-1} X_1^N))\cdots \tr (a_{q_r}^{(r)} T_{q_r}(K_r^{-1} X_r^N))\right]\right. \notag \\
& \hspace{13em}\left.-\sum_{k=0}^{m-1} \frac{\mu_k\left(a_{q_1}^{(1)} T_{q_1}(K_1^{-1}\cdot ) , \dots , a_{q_r}^{(r)} T_{q_r}(K_r^{-1} \cdot)\right)}{N^k}\Bigg)\right| \notag \\
& \leq \sum_{q_1, ..., q_r} |a_{q_1}^{(1)}|\cdots |a_{q_r}^{(r)}|
\Bigg|\E \left[\tr T_{q_1}(K_1^{-1}X_1^N)\cdots \tr T_{q_r}(K_r^{-1}X_r^N)\right]\notag \\
&\hspace{14em}-\sum_{k=0}^{m-1} \frac{\mu_k (T_{q_1}(K_1^{-1} \cdot ),..., T_{q_r}(K_r^{-1}\cdot))}{N^k}\Bigg| \notag \\
& \leq \frac{1}{N^m} \sum_{q_1, ..., q_r} |a_{q_1}^{(1)}|\cdots |a_{q_r}^{(r)}| 
\left((C\bar{p})^m + \frac{(C\bar{p})^{2m}}{m!}\right),
\end{align*}
where we applied Proposition \ref{prop:ExpansionPolyUnitary} in the last inequality, with $p=\tilde{q_1}q_1+\cdots +\tilde{q_r}q_r$ and $\bar{p}=p(1+\log p)$.

Now, we note that for any $u \in (0,1)$, observing that $\log x \le x$ for all $x>0$,
\begin{align*}
\bar{p} 
=p(1+\log p) 
=p\cdot \frac{1}{u} \log ((\e p)^u)
\leq \frac{\e^u}{u} p^{1+u}.
\end{align*}
Thus, we have
\begin{align*}
&\frac{1}{N^m} \sum_{q_1, ..., q_r} |a_{q_1}^{(1)}|\cdots |a_{q_r}^{(r)}| 
\left((C\bar{p})^m + \frac{(C\bar{p})^{2m}}{m!}\right) \notag \\
&\leq \frac{1}{N^m} \sum_{q_1, ..., q_r} |a_{q_1}^{(1)}|\cdots |a_{q_r}^{(r)}| \left(\frac{(C\e^u)^m}{u^m} p^{\lceil (1+u)m \rceil} + \frac{(C\e^u)^{2m}}{m! u^{2m}}p^{\lceil 2(1+u)m \rceil}\right) \notag \\
&\leq \frac{C^{m+r}}{N^m} \Bigg( \frac{(\tilde{q}_1+\cdots +\tilde{q}_r)^{\lceil (1+u)m \rceil}}{u^m} \|F\|_{C^{\lceil (1+u)m \rceil}([0,2\pi]^r)} \notag \\
& \hspace{12em} + \frac{(\tilde{q}_1+\cdots +\tilde{q}_r)^{\lceil 2(1+u)m \rceil}}{m! u^{2m}} \|F\|_{C^{\lceil 2(1+u)m \rceil}([0,2\pi]^r)} \Bigg) ,
\end{align*}
with some new universal constant $C>0$, where we used \eqref{ineq:FourierCoeff} in the last inequality.
This completes the proof.

\end{proof}

\begin{remark}\label{remark:C^M_unitary}
By the same reasoning as in Theorem \ref{thm:ExpansionSmoothGUE} and Remark \ref{remark:C^M_GUE}, one can also check that by chosing $u$ sufficiently small so that $\lceil 2(1+u)m \rceil =2m+1$, one can obtain
\begin{align}
|\mu_m(h_1,...,h_r)|\leq C \|h_1^{(1)}(x_1)\cdots h_r^{(1)}(x_r)\|_{C^{2m+1} ([-K_1,K_1]\times \cdots \times [-K_r,K_r])},
\end{align}
where $C$ is a constant depending only on $m,r$. 
So, $\mu_m$ is well-defined on $C^{2m+2}(\R)^r$, and the asymptotic expansion \eqref{ineq:ExpansionSmoothUnitary} in Theorem \ref{thm:ExpansionSmoothUnitary} also holds for $h_1,...,h_r$ in the $C^{2m+2}$ class.
\end{remark}

\subsection{The Orthogonal and symplectic case}\label{sec:orthogonal-symplectic}
In this section, we briefly discuss the extension of the asymptotic expansion results, Theorem \ref{thm:ExpansionSmoothUnitary}, to the cases of Haar-distributed orthogonal and symplectic matrices.
The proof strategy is exactly the same as that of the unitary case -- one can get the asymptotic expansion for polynomial test functions, an analogue of Proposition \ref{prop:ExpansionPolyUnitary}, by reducing the multi-trace case to the single-trace case treated in \cite{VH_newapproach2} as in the same way as the unitary case, and then extend it to smooth test functions by approximating them with polynomials via Chebyshev expansions.
So we only state the results here without going into the proofs.

Let $\bm{U}^N = (U_1^N, \ldots, U_d^N)$ and $\bm{V}^N = (V_1^N, \ldots, V_d^N)$ be independent Haar distributed orthogonal and Symplectic matrices in $\Ort (N)$ and $\Sp (N)$ of dimension $N$, respectively, and let $\bm{u} = (u_1, \ldots, u_d)$ be a free Haar unitaries. 
Let $r \in \N$ be fixed. 
We will also fix self-adjoint non-commutative polynomials $P_i \in 
\mathbb{C}\langle x_1, \ldots, x_d, x_1^*,\ldots, x_d^*\rangle$ of degree $\tilde{q}_i$ for each $i=1,...,r$. 
We will use the following notations for simplicity, 
\[
X_i^N := P_i(\bm{U}^N,\bm{U}^{N*}), \qquad Y_i^N := P_i(\bm{V}^N,\bm{V}^{N*}), \qquad X_i^F := P_i(\bm{u},\bm{u}^*) \qquad i=1,...,r.
\]
We also define $K_i \ (i=1,\dots,r)$ as in Section \ref{sec:unitary}.

One can get the asymptotic expansions as analogues of Theorem \ref{thm:ExpansionSmoothUnitary} for both $\Ort(N)$ and $\Sp(N)$ matrices, that is, for both $\mathbb{E}[\tr h_1 (X_1^N)\cdots \tr h_r(X_r^N)]$ and $\mathbb{E}[\tr h_1 (Y_1^N)\cdots \tr h_r(Y_r^N)]$ individually.
However, as \cite{VH_newapproach2} points out, by the supersymmetric duality (see \cite{MP24_MR4700374}), the asymptotic expansions for $X_i^N$ and $Y_i^N$ can be controlled simultaneously, so instead of stating the individual estimates separately, we record the following joint bound.

\begin{theorem}[Asymptotic expansions for multi-trace $O(N)$ and $\Sp (N)$ matrices]\label{thm:ExpansionSmoothO-Sp}
For any $m \in \Z_{\geq 0}$, there exists an $r$-linear functional $\mu_m$ on $C^\infty (\R)^r$ such that the following holds with some universal constant $C>0$:\\
fix any $r$-tuple of smooth functions $h_i \in C^\infty (\R)\ (i=1,...,r)$, 
and define 
\begin{align*}
&f_i(\theta):=h_i (K_i \cos \theta)
\quad \text{with} \quad
K_i=\|P_i\|_{C^* (F_d)} , \quad i=1,...,r,\\
&\quad \text{and } \quad F(x_1,\dots ,x_r):= f_1^{(1)}(x_1)\cdots f_r^{(1)}(x_r);
\end{align*} 
then, for every $m, N \in \mathbb{N}$, and for any $u \in (0,1)$, we have
\begin{align}\label{ineq:ExpansionSmoothO-Sp}
&\left|\E [\tr h_1 (X_1^N)\cdots \tr h_r(X_r^N)]-\sum_{k=0}^{m-1} \frac{\nu_k(h_1, \dots ,h_r)}{N^k}\right| 
\vee \left|\E [\tr h_1 (Y_1^N)\cdots \tr h_r(Y_r^N)]-\sum_{k=0}^{m-1} \frac{\nu_k(h_1, \dots ,h_r)}{(-2N)^k}\right|
\notag \\
&\leq \frac{C^{m+r}}{N^m} \Bigg( \frac{(\tilde{q}_1+\cdots +\tilde{q}_r)^{\lceil m^\prime \rceil}}{u^m} \|F\|_{C^{\lceil m^\prime \rceil}([0,2\pi]^r)} + \frac{(\tilde{q}_1+\cdots +\tilde{q}_r)^{\lceil 2m^\prime \rceil}}{m! u^{2m}} \|F\|_{C^{\lceil 2m^\prime \rceil}([0,2\pi]^r)} \Bigg) ,
\end{align}
where $m^\prime := (1+u)m$.
\end{theorem}

\section{Applications of the asymptotic expansions}\label{sec:Application}

\subsection{Higher-order asymptotic vanishing of cumulants}
In this section, as an application of the asymptotic expansion results, 
(e.g. Theorem \ref{thm:ExpansionSmoothGUE}, Theorem \ref{thm:ExpansionSmoothUnitary}), we show the higher-order asymptotic vanishing of cumulants for traces of functions in classical random matrix ensembles.

We begin by recalling the definition of cumulants, and refer to \cite{Zvonkin97_MR1492512}, \cite{Rota64_MR174487}, \cite{Stanley97_MR1442260}.
For random variables $Y_1,\dots,Y_r$ and a partition $\pi=\{B_1,...,B_\pi\}\in\Part (r)$, the classical cumulant $C_\pi (Y_1,\dots,Y_r)$ is defined by
\begin{align*}
C_\pi (Y_1,\dots,Y_r)
:= \sum_{\pi^\prime \leq \pi} \E_{\pi^\prime}[Y_1,\dots,Y_r] \Moeb(\pi^\prime,\pi),
\end{align*}
and for $\pi=1_r$, we write $C_r(Y_1,\dots,Y_r):=C_{1_r}(Y_1,\dots,Y_r)$, where
\begin{align*}
\E_{\pi}[Y_1,\dots,Y_r]
:= \prod_{j=1}^{\#\pi} \E\left[\prod_{i\in B_j} Y_i\right].
\end{align*}
Then we have the moment-cumulant formula
\begin{align*}
\E_\pi [Y_1\cdots Y_r]
= \sum_{\pi^\prime \leq \pi} C_{\pi^\prime} (Y_1,\dots,Y_r).
\end{align*}

There are several alternatives to define cumulants, for instance, via generating functions, and one is 
\begin{align*}
C_l(Y_{i(1)},...,Y_{i(l)})=\frac{\partial^l}{\partial t_{i(1)}\cdots \partial t_{i(l)}}\log \E\left[\exp\left(t_1 Y_1+\cdots +t_r Y_r\right)\right]\Bigg|_{t_1=\cdots =t_r=0},
\end{align*}
for any $l \in \N$ and any indices $i(1),...,i(l) \in \{1,...,r\}$.
With this definition, one knows the relation between moments and cumulants in terms of their formal generating functions:
\begin{align}
\log \E[\exp (t Y)] 
=\log \left( \sum_{k\geq 0} \E[Y^k] \frac{t^k}{k!} \right)
= \sum_{k\geq 1} C_k(Y,...,Y) \frac{t^k}{k!},
\end{align}
and we also have the multivariate version of the above relation:
\begin{align}\label{eq:CumulantGenerating_multi}
\log \E\left[\exp\left(t_1 Y_1+\cdots + t_r Y_r\right)\right]
&=\log \left( \sum_{\bm{k} \in \Z_{\geq 0}^r} \E[Y_1^{k_1}\cdots Y_r^{k_r}] \frac{\bm{t}^{\bm{k}}}{\bm{k}!} \right) \notag \\
&=\sum_{\bm{k} \in \Z_{\geq 0}^r \setminus \{\bm{0}\}} C_{\bm{k}}(Y_1,...,Y_r) \frac{\bm{t}^{\bm{k}}}{\bm{k}!},
\end{align}
where for $\bm{k}=(k_1,...,k_r) \in \Z_{\geq 0}^r \setminus \{ \bm{0}\}$, 
$\bm{t}^{\bm{k}}:=t_1^{k_1}\cdots t_r^{k_r}$,
$\bm{k}!:=k_1!\cdots k_r!$, and
$C_{\bm{k}}(Y_1,...,Y_r)$ is the cumulant defined by
\begin{align*}
C_{\bm{k}}(Y_1,...,Y_r):=C_{k_1+\cdots +k_r}(\underbrace{Y_1,...,Y_1}_{k_1 \text{ times}},...,\underbrace{Y_r,...,Y_r}_{k_r \text{ times}}).
\end{align*}
This definition is equivalent to the previous one defined by the moment-cumulant formula, since expanding the logarithm of the moment generating function yields the moment–cumulant relations. 
(See for details on the computation, e.g. section 6 in \cite{RotaShen00_MR1779783}.)

We now turn our attention to the higher-order asymptotic vanishing of cumulants for traces of functions of classical random matrices--GUE/GOE/GSE and Haar-distributed random matrices in $\U(N), \Ort(N), \Sp(N)$--considered in the previous sections. 
One knows that for 
those matrices,
the $r$-th order cumulant of unnormalized traces in polynomials of those random matrices decays with the order $N^{2-r}$ as $N \to \infty$, and here we restate this fact for our settings. 
For proofs of the results below, we refer to 
\cite{Collins2003}, \cite{MS06} for the Gaussian, orthogonal and symplectic cases, and to \cite{CMSS07} for a more systematic treatment (higher order freeness) in the unitary case. 
Note that unlike the first order asymptotics that are the same irrespective of the symmetry class (unitary, orthogonal, symplectic -- always asymptotic freeness), 
higher order cumulants rely heavily on the symmetry class, see \cite{Redelmeier14_MR3217665} and \cite{MP13_MR3098908}.

\begin{theorem}[Asymptotic vanishing of cumulants for polynomial test functions]\label{thm:CumulantPoly}
Take the same setting and notations as in Section \ref{sec:GUE}, \ref{sec:GOE-GSE}, \ref{sec:unitary}, or \ref{sec:orthogonal-symplectic}, and let $X_1^N,...,X_r^N$ be self-adjoint polynomials of either 
GUEs, GOEs, GSEs, or Haar-distributed matrices in $\U(N),\Ort(N), \Sp(N)$.
For any polynomials $h_i \in \mathcal{P} \ (i=1,...,r)$, we define $h_i(X_i^N) \ (i=1,...,r)$.
Then, 
the limit
\begin{align}
\lim_{N \to \infty} N^{-2} C_r(N\Tr h_{1}(X_{1}^{N}),...,N \Tr h_{r}(X_{r}^{N}))
\end{align}
exists.
\end{theorem}

Now we extend it to the case of smooth test functions by using the asymptotic expansion results. 

\begin{theorem}[Higher-order asymptotic vanishing of cumulants for smooth test functions]\label{thm:CumulantSmooth}
Take the same setting and notations as in Section \ref{sec:GUE}, \ref{sec:GOE-GSE}, \ref{sec:unitary}, or \ref{sec:orthogonal-symplectic}, and let $X_1^N,...,X_r^N$ be self-adjoint polynomials of either 
GUEs, GOEs, GSEs, or Haar-distributed matrices in $\U(N),\Ort(N), \Sp(N)$.
In the Gaussian ensemble cases, for any bounded smooth functions $h_i \in C^{4r-1} \ (i=1,...,r)$, whereas in the Haar-distributed matrix cases, for any smooth functions $h_i \in C^{4r} \ (i=1,...,r)$, we define $h_i(X_i^N) \ (i=1,...,r)$.
Then,  
the limit
\begin{align}\label{eq:CumulantSmooth_limit}
\lim_{N \to \infty} N^{-2} C_r(N\Tr h_{1}(X_{1}^{N}),...,N \Tr h_{r}(X_{r}^{N}))
\end{align}
exists.
\end{theorem}

\begin{remark}
    We only give the proof in the GUE case, as Gaussian ensembles with different symmetry classes (orthogonal, symplectic), and Haar distributed ensembles (unitary, orthogonal symplectic) can be shown in the same way.
    Note, however that unlike the expectation of traces, higher order cumulants depend on the symmetry class in Theorem \ref{thm:CumulantPoly} and therefore, also in Theorem \ref{thm:CumulantSmooth}. The limits are governed by higher order freeness (real, orthogonal, symplectic). 
    For details, we refer to 
    \cite{Redelmeier14_MR3217665} and \cite{MP13_MR3098908}.
\end{remark}

\begin{proof}
By the definition of cumulants, we have
\begin{align}\label{eq:cumulant_moment_expansion}
&N^{-2} C_r(N\Tr h_{1}(X_{1}^{N}),...,N \Tr h_{r}(X_{r}^{N})) \notag \\
&=N^{2r-2} C_r(\tr h_{1}(X_{1}^{N}),...,\tr h_{r}(X_{r}^{N})) \notag \\
&=N^{2r-2} \sum_{\pi \in \Part (r)} \E_{\pi} \left[\tr h_{1}(X_{1}^{N}),...,\tr h_{r}(X_{r}^{N})\right] \Moeb (\pi, 1_r).
\end{align}
Let $\nu_l \ (0\leq l \leq m-1)$ be the $r$-linear functionals on $C^{2m+1} (\R)^r$ or $\mathcal{P}^r$ given in Proposition \ref{prop:ExpansionPolyGUE} and Theorem \ref{thm:ExpansionSmoothGUE}. 
For any $B\subset\{1,\dots,r\}$, we define $a_l^B \ (0 \leq l \leq m-1)$ as 
\begin{align}
a_l^B \big(h_1,...,h_r \big):=\nu_l(g_1,\dots,g_r),
\end{align}
where $g_i=h_i$ for $i\in B$ and $g_i=1$ for $i\notin B$.

For any partition $\pi=\{B_1,...,B_{\# \pi}\} \in \Part (r)$, Proposition \ref{prop:ExpansionPolyGUE} and Theorem \ref{thm:ExpansionSmoothGUE} yields that for
an $r$-tuple $h=(h_1,\ldots , h_r)$ that is either
bounded smooth functions in $C^{2m+1}(\R)^r$ or an element of  $\mathcal{P}^r$, we have
\begin{align}
&\E_{\pi} \left[\tr h_1(X_1^{N}),...,\tr h_r(X_r^{N})\right] \notag \\
&=\prod_{j=1}^{\# \pi} \E \left[\prod_{i \in B_j} \tr h_i(X_i^{N})\right] \notag \\
&=\prod_{j=1}^{\# \pi} \left(\sum_{l=0}^{m-1} \frac{a_l^{B_j} (h_1,...,h_r)}{N^l} + \frac{a_m^{B_j}(h_1,...,h_r)(N)}{N^m}\right) \notag \\
&=\sum_{l=0}^{m-1} \frac{1}{N^l} \sum_{\substack{0 \leq l_1,...,l_{\# \pi} \leq m-1 \\ l_1+\cdots +l_{\# \pi}=l}} \prod_{j=1}^{\# \pi} a_{l_j}^{B_j} (h_1,...,h_r) +\sum_{l=m}^{m \times \# \pi} \frac{1}{N^l} \sum_{\substack{l_1,...,l_{\# \pi} \geq 0 \\ l_1+\cdots +l_{\# \pi}=l}} \prod_{j=1}^{\# \pi} a_{l_j}^{B_j} (h_1,...,h_r),
\end{align}
where $a_m^{B_j}(h_1,...,h_r)(N)$ is a bounded function in $N$.
Specifically, we have
\begin{align}\label{ineq:a-bound}
\limsup_{N \to \infty} |a_m^{B_j}(h_1,...,h_r)(N)| \leq C \|h_1^{(1)}(x_1)\cdots h_r^{(1)}(x_r)\|_{C^{2m}([-K_1,K_1]\times \cdots \times [-K_r,K_r])},
\end{align}
with some constant $C>0$ depending on $m, r$ for any $j$.
Combining with \eqref{eq:cumulant_moment_expansion}, we have 
\begin{align}\label{eq:cumulant_N_expansion}
&N^{-2} C_r(N\Tr h_1(X_1^{N}),...,N \Tr h_r(X_r^{N})) \notag \\
&=\sum_{l=0}^{m-1} N^{2r-2-l} \sum_{\substack{\pi = \{B_1,...,B_{\# \pi}\}\\ \in \Part (r)}} \Moeb (\pi, 1_r) \sum_{\substack{l_1,...,l_{\# \pi} \geq 0 \\ l_1+\cdots +l_{\# \pi}=l}} \prod_{j=1}^{\# \pi} a_{l_j}^{B_j} (h_1,...,h_r)
+ O\left(N^{2r-2-m}\right).
\end{align}
As we saw \eqref{ineq:nu_m_Bound} in Theorem \ref{thm:ExpansionSmoothGUE}, the functionals $\nu_l \ (0 \leq l \leq m-1)$ satisfy 
\begin{align*}
|\nu_l (h_1,...,h_r)| \leq C \|h_1^{(1)}(x_1)\cdots h_r^{(1)}(x_r)\|_{C^{2l}([-K_1,K_1]\times \cdots \times [-K_r,K_r])},
\end{align*}
where 
$C>0$ is a constant depending only on $l,r$, so the same bound holds for $a_l^B$ for any $B\subset\{1,\dots,r\}$.
Choose sequences of polynomials $\{g_{n,i}\}_{n=1}^\infty$ such that $g_{n,i}\to h_i$ in the $C^{2m+1}([-K_i,K_i])$ norm as $n\to\infty$ for each $i=1,...,r$.
Then, by the above bound combined with the multi-linearity, we have
\begin{align*}
\lim_{n \to \infty} a_l^B (g_{n,1},...,g_{n,r}) = a_l^B (h_1,...,h_r),
\end{align*}
for any $0 \leq l \leq m-1$ and any $B \subset \{1,...,r\}$.

We now take $m$ to be $2r-1$. Theorem \ref{thm:CumulantPoly} yields that for any polynomials $g_{n,i} \in \mathcal{P} \ (i=1,...,r)$, we have for each $0\leq l \leq 2r-3$ the coefficient of $N^{2r-2-l}$ in \eqref{eq:cumulant_N_expansion} vanishes, i.e.,
\begin{align}\label{eq:CoeffVanish}
\sum_{\substack{\pi = \{B_1,...,B_{\# \pi}\}\\ \in \Part (r)}} \Moeb (\pi, 1_r) \sum_{\substack{l_1,...,l_{\# \pi} \geq 0 \\ l_1+\cdots +l_{\# \pi}=l}} \prod_{j=1}^{\# \pi} a_{l_j}^{B_j} (g_{n,1},...,g_{n,r})=0.
\end{align}
Thus, by taking the limit $n \to \infty$, \eqref{eq:CoeffVanish} also holds for any bounded smooth functions $h_i \in C^{2m+1} (\R)$. Alternatively, this extension can be justified by viewing the left-hand side of \eqref{eq:CoeffVanish} as a bounded multi-linear functional. Since the functional vanishes on the dense subspace of polynomials $\mathcal{P}$, Lemma \ref{lemma:extensionBoundedMultilinear} implies it must vanish on the entire space of smooth functions.

Therefore, we conclude that for any bounded smooth functions $h_i \in C^{2(2r-1)+1} (\R) \ (i=1,...,r)$, the limit
\begin{align*}
\lim_{N \to \infty} N^{-2} C_r(N\Tr h_{1}(X_{1}^{N}),...,N \Tr h_{r}(X_{r}^{N}))
\end{align*}
exists, and 
\begin{align*}
&\lim_{N \to \infty} N^{-2} C_r(N\Tr h_{1}(X_{1}^{N}),...,N \Tr h_{r}(X_{r}^{N})) \notag \\
&=\sum_{\substack{\pi = \{B_1,...,B_{\# \pi}\}\\ \in \Part (r)}} \Moeb (\pi, 1_r) \sum_{\substack{l_1,...,l_{\# \pi} \geq 0 \\ l_1+\cdots +l_{\# \pi}=2r-2}} \prod_{j=1}^{\# \pi} a_{l_j}^{B_j} (h_{1},...,h_{r}) \notag \\
&= \lim_{n \to \infty} \sum_{\substack{\pi = \{B_1,...,B_{\# \pi}\}\\ \in \Part (r)}} \Moeb (\pi, 1_r) \sum_{\substack{l_1,...,l_{\# \pi} \geq 0 \\ l_1+\cdots +l_{\# \pi}=2r-2}} \prod_{j=1}^{\# \pi} a_{l_j}^{B_j} (g_{n,1},...,g_{n,r}) \notag \\
&= \lim_{n \to \infty} \lim_{N \to \infty} N^{-2} C_r(N\Tr g_{n,1}(X_{1}^{N}),...,N \Tr g_{n,r}(X_{r}^{N})),
\end{align*}
which completes the proof.
\end{proof}

\subsection{Central limit theorems}
As a corollary of Theorem \ref{thm:CumulantSmooth}, we get the Central Limit Theorems (CLTs) for traces of functions of classical matrix ensembles. CLTs in the context of statistics of random matrices involving classical ensembles have been the subject of intense study.
For linear statistics of single matrix models, this subject has been well-studied, see for example \cite{Johansson1998_MR1487983}, \cite{Guionnet2002_MR1899457}, \cite{LP09_MR2561434}. 

Sharp conditions regarding the regularity of test functions have been extensively studied.
For Haar-distributed matrices, this topic is discussed, for example, in \cite{Soshnikov00_MR1797877, Silverstein90_MR1062064, JL21_MR4348683, Silverstein22_MR4480828, CJ24_MR4718383}. For Wigner matrices (including GUE/GOE/GSE), related topics have been studied in \cite{Shcherbina11_MR2829615, SW13_MR3116567, BD17_MR3556288, BaoHe23_MR4677731, CES23_MR4551555, Reker24_MR4842843}.
Many of these papers show that regularity conditions slightly weaker than $C^1$ (specifically of Sobolev type $H^{1/2}$) are sufficient.

However, these works primarily concern linear spectral statistics or models where the matrix is modified by a constant matrix, rather than the non-commutative functional calculus of multiple random matrices.

In the setting of non-commutative polynomials and multiple random matrices, the convergence and joint fluctuations for fixed-degree polynomials of classical random matrix ensembles follow from the combinatorial theory of second-order freeness; see, for example, \cite{MS06_MR2216446, MSS07_MR2294222, CMSS07, MP13_MR3098908, Redelmeier14_MR3217665}. In particular, recent results on the structure of the fluctuation moments of Wigner matrices were obtained in \cite{CMPS22_MR4414193}.

In this section, we focus on extending these results beyond polynomials to smooth functions.
After completing the results of this section, we figured out that the paper \cite{DM23_MR4552697} by Diaz and Mingo is relevant to our manuscript, in the sense that it establishes criteria for CLTs to hold in the multi-matrix context with smooth functions -- in the framework of bounded Fréchet variation for $C^1$-functions.
As far as we can tell, although some approximation steps bear similarities with our results, the results of \cite{DM23_MR4552697} do not allow to rederive our results directly. 

We start with $C^\infty$ test functions and observe that our results extend even further to a context of finite regularity.

\begin{theorem}
[Central Limit Theorem for smooth test functions]\label{thm:CLT}
Take the same setting and notations as in Section \ref{sec:GUE}, \ref{sec:GOE-GSE}, \ref{sec:unitary}, or \ref{sec:orthogonal-symplectic}, and let $X_1^N,...,X_r^N$ be self-adjoint polynomials of 
either GUEs, GOEs, GSEs, or Haar-distributed matrices in $\U(N),\Ort(N), \Sp(N)$.
For any smooth functions $h_i \in C^\infty (\R) \ (i=1,...,r)$, (and in the Gaussian ensemble cases, we additionally assume boundedness), the random vector
\begin{align*}
\left(\Tr h_i(X_i^{N})-\E[\Tr h_i(X_i^{N})]\right)_{i=1}^r
\end{align*}
converges in distribution as $N \to \infty$ to a centered Gaussian random vector $\mathcal{N}(\bm{0}, \Sigma)$ whose covariance is given by
\begin{align*}
\Sigma_{i,j}
:=\lim_{N \to \infty} \E\Big[\left(\Tr h_i(X_i^{N})-\E[\Tr h_i(X_i^{N})]\right) \cdot \left(\Tr h_j(X_j^{N})-\E[\Tr h_j(X_j^{N})]\right)\Big].
\end{align*}
\end{theorem}

\begin{proof}
Let $Y^N_1,...,Y^N_r$ be the random variables defined by 
\begin{align*}
Y^N_i:=\Tr h_i(X_i^{N})-\E[\Tr h_i(X_i^{N})], \quad i=1,...,r,
\end{align*}
where we note that $Y^N_i$ is centered
.

By Theorem \ref{thm:CumulantSmooth}, we know that for any $k \geq 3$ and $1 \leq i_1,...,i_k \leq r$, the $k$-th order cumulant $C_k(Y^N_{i_1},...,Y^N_{i_k})$ vanishes as $N \to \infty$,
\begin{align*}
\lim_{N \to \infty} C_k(Y^N_{i_1},...,Y^N_{i_k})=0,
\end{align*}
and for any $1 \leq i,j \leq r$, the second order cumulant $C_2(Y^N_i,Y^N_j)$ converges to some finite value,
\begin{align*}
\lim_{N \to \infty} \operatorname{Cov}(Y^N_i,Y^N_j)
=\lim_{N \to \infty} C_2(Y^N_i,Y^N_j)=:\Sigma_{i,j} < \infty.
\end{align*}
With the fact that a distribution with only first and second cumulants nonzero is Gaussian, the desired CLT follows.
\end{proof}

We can further extend Theorem \ref{thm:CLT} to the case where the test functions are in some finite $C^L$ class.

\begin{theorem}
[Extension of CLT to $C^L$-functions]\label{thm:CLT-Cfinite}
Take the same setting and notations as in Section \ref{sec:GUE}, \ref{sec:GOE-GSE}, \ref{sec:unitary}, or \ref{sec:orthogonal-symplectic}, and let $X_1^N,...,X_r^N$ be self-adjoint polynomials of 
either GUEs, GOEs, GSEs, or Haar-distributed matrices in $\U(N),\Ort(N), \Sp(N)$.
There exists some positive integer $L$ such that the following holds:\\
For any smooth functions $h_i \in C^L (\R) \ (i=1,...,r)$, (and in the Gaussian ensemble cases, we additionally assume boundedness), the random vector
\begin{align*}
\left(\Tr h_i(X_i^{N})-\E[\Tr h_i(X_i^{N})]\right)_{i=1}^r
\end{align*}
converges in distribution as $N \to \infty$ to a centered Gaussian random vector $\mathcal{N}(\bm{0}, \Sigma)$ whose covariance is given by
\begin{align*}
\Sigma_{i,j}
:=\lim_{N \to \infty} \E\Big[\left(\Tr h_i(X_i^{N})-\E[\Tr h_i(X_i^{N})]\right) \cdot \left(\Tr h_j(X_j^{N})-\E[\Tr h_j(X_j^{N})]\right)\Big].
\end{align*}
In particular, we can take $L=5$ in the Gaussian ensemble cases, and $L=6$ in the Haar-distributed matrix cases.
\end{theorem}

\begin{proof}
For any positive integer $L$ and any $\varepsilon >0$, by a standard density argument using mollifiers, we can approximate the $C^L$ functions $h_i$ by smooth functions $h_i^\varepsilon \in C^\infty$, so that 
\begin{align*}
\|h_i - h_i^\varepsilon\|_{C^L([-K_i,K_i])} < \varepsilon.
\end{align*}
for each $i=1,...,r$. 
(Specifically, one can check this fact observing $h_i^\varepsilon := h_i * \rho_\delta$ with a suitable mollifier $\rho_\delta$. 
We also note that in the GUE case, since $h_i$ is bounded, this approximation can be done while preserving the uniform bound $\|\cdot\|_{(-\infty, \infty)}$.) 

Let us denote by $\bm{Y}^N(h_1, ..., h_r)$ the random vector defined by
\begin{align}\label{eq:def-Y}
\bm{Y}^N(h_1, ..., h_r)
&=(Y_1^N(h_1),...,Y_r^N(h_r)) \notag \\
&:= (\Tr h_1(X_1^N)-\E[\Tr h_1(X_1^N)], ..., \Tr h_r(X_r^N)-\E[\Tr h_r(X_r^N)]).
\end{align}

Since we have 
\begin{align*}
\bm{Y}^N(h_1, ..., h_r) - \bm{Y}^N(h_1^\varepsilon, ..., h_r^\varepsilon) 
= \bm{Y}^N((h_1 - h_1^\varepsilon), ..., (h_r - h_r^\varepsilon)),
\end{align*}
to show that $\bm{Y}^N(h_1, ..., h_r)$ converges in distribution to $\mathcal{N}(\bm{0}, \Sigma)$ as $N \to \infty$, it suffices to prove the following:\\
(I) the vanishing of the error term, 
\begin{align}\label{conv:error-vanish}
\bm{Y}^N((h_1 - h_1^\varepsilon), ..., (h_r - h_r^\varepsilon)) \xrightarrow[\varepsilon \to 0]{} \bm{0}
\end{align}
in distribution, uniformly in $N$ for each $i=1,...,r$\\
(II) the convergence of the approximated vector to a Gaussian vector, namely,
\begin{align}\label{conv:approx-CLT}
\bm{Y}^N(h_1^\varepsilon, ..., h_r^\varepsilon) \xrightarrow[N \to \infty]{} \mathcal{N}(\bm{0}, \Sigma_{h_1^\varepsilon,...,h_r^\varepsilon}),
\end{align}
where $\Sigma_{h_1^\varepsilon,...,h_r^\varepsilon}$ is the covariance matrix, and in addition, the covariance matrix $\Sigma_{h_1^\varepsilon,...,h_r^\varepsilon}$ converges to $\Sigma$ as $\varepsilon \to 0$.

(I) is shown as follows.
Since for any $\delta>0$ and each $i=1,...,r$, we have
\begin{align*}
\limsup_{N \in \mathbb{N}} \mathbb{P} \left( |Y_i^N (h_i-h_i^\varepsilon)| > \delta \right)
\leq \frac{1}{\delta^2} \limsup_{N \in \mathbb{N}} \mathrm{Var}(Y_i^N (h_i-h_i^\varepsilon)),
\end{align*}
by Chebyshev's inequality, it is enough to show that the variance on the right-hand side converges to zero as $\varepsilon \to 0$ uniformly in $N$ for each $i=1,...,r$.
One can prove it in a similar manner to the proof of Theorem \ref{thm:CumulantSmooth}. We only give the proof 
in the GUE case, as Gaussian ensembles with different symmetry classes, and Haar distributed ensembles can be shown in the same way.

When $h_i \in C^5(\R)$ (the case $L=5$)
, by the $1/N$-expansion in Theorem \ref{thm:ExpansionSmoothGUE}
, we can see that the variance is bounded uniformly in $N$ as
\begin{align*}
&\limsup_{N \in \N} \mathrm{Var}(Y_i^N(h_i - h_i^\varepsilon))
= \limsup_{N \in \N} C_2 (\Tr (h_i - h_i^\varepsilon)(X_i^N), \Tr (h_i - h_i^\varepsilon)(X_i^N))\\
&= \limsup_{N \in \N} N^{2} \left\{\E \left[\tr (h_i - h_i^\varepsilon)(X_i^N) \tr (h_i - h_i^\varepsilon)(X_i^N)\right] - \E \left[\tr (h_i - h_i^\varepsilon)(X_i^N)\right] \E \left[\tr (h_i - h_i^\varepsilon)(X_i^N)\right]\right\}\\
&= \limsup_{N \in \N} N^{2} \left\{\left(\nu_0(h_i - h_i^\varepsilon, h_i - h_i^\varepsilon)+\frac{\nu_1(h_i - h_i^\varepsilon, h_i - h_i^\varepsilon)}{N}+\frac{a_2(h_i - h_i^\varepsilon, h_i - h_i^\varepsilon)(N)}{N^2}\right)\right. \\
& \qquad \qquad \qquad \left.-\left(\nu_0(h_i - h_i^\varepsilon)+\frac{\nu_1(h_i - h_i^\varepsilon)}{N}+\frac{a_2(h_i - h_i^\varepsilon)(N)}{N^2}\right)^2\right\}\\
&= \limsup_{N \in \N} \left\{\left(\nu_0(h_i - h_i^\varepsilon, h_i - h_i^\varepsilon)-\nu_0(h_i - h_i^\varepsilon)^2 \right)N^2 \right.\\
& \qquad \qquad \qquad +\left(\nu_1(h_i - h_i^\varepsilon, h_i - h_i^\varepsilon)-2\nu_0(h_i - h_i^\varepsilon)\nu_1(h_i - h_i^\varepsilon)\right)N\\
& \qquad \qquad \qquad \left. + a_2(h_i - h_i^\varepsilon, h_i - h_i^\varepsilon)(N)-2\nu_0(h_i - h_i^\varepsilon)a_2(h_i - h_i^\varepsilon)(N)-\nu_1(h_i - h_i^\varepsilon)^2+O(1/N)\right\}\\
&\leq C \| (h_i - h_i^\varepsilon)^{(1)} \|_{C^{4}([-K_i,K_i])}^2
\leq C \varepsilon^2,
\end{align*}
where $C$ is a universal constant independent of $N, \varepsilon$, and $\nu_i$s are those in Theorem \ref{thm:ExpansionSmoothGUE} and $a_2(h_1,...,h_r)(N):=a_2^{[r]}(h_1,...,h_r)(N)$s are those in the proof of Theorem \ref{thm:CumulantSmooth}.
Here, the last two inequalities follow, from Theorem \ref{thm:CumulantSmooth}, the coefficients of $N^2$ and $N$ vanish, and from \eqref{ineq:nu_m_Bound} and \eqref{ineq:a-bound}, the remaining term is bounded by $\| (h_i - h_i^\varepsilon)^{(1)} \|_{C^{4}([-K_i,K_i])}^2$, which completes the proof of \eqref{conv:error-vanish}.

(II) is a consequence of Theorem \ref{thm:CLT}.
By Theorem \ref{thm:CLT}, we know that $\bm{Y}^N(h_1^\varepsilon, ..., h_r^\varepsilon)$ converges in distribution as $N \to \infty$ to a centered Gaussian random vector $\mathcal{N}(\bm{0}, \Sigma_{h_1^\varepsilon,...,h_r^\varepsilon})$. 
Also the covariance matrix $\Sigma_{h_1^\varepsilon,...,h_r^\varepsilon}$ is 
\begin{align*}
&(\Sigma_{h_1^\varepsilon,...,h_r^\varepsilon})_{i,j}
=\lim_{N \to \infty} \operatorname{Cov}(Y_i^N(h_i^\varepsilon), Y_j^N(h_j^\varepsilon))
=\lim_{N \to \infty} C_2(\Tr h_i^\varepsilon(X_i^{N}), \Tr h_j^\varepsilon(X_j^{N}))\\
&\quad \xrightarrow[\varepsilon \to 0]{} \lim_{N \to \infty} C_2(\Tr h_i(X_i^{N}), \Tr h_j(X_j^{N}))
=:\Sigma_{i,j} < \infty,
\end{align*}
where the boundedness and the convergence as $\varepsilon \to 0$ follow from the same argument as in (I) above, 
which completes the proof of (II).
\end{proof}

\subsection{Matrix integrals}

We consider the matrix integrals of the form
\begin{align}
I_N (V)
:=\int
\e^{N \Tr V(U_1,...,U_d,U_1^*,...,U_d^*)} \d U_1\cdots \d U_d,
\end{align}
and we define the free energy as
\begin{align}
F_N (V)
:=\frac{1}{N^2} \log I_N (V),
\end{align}
where $V$ is a self-adjoint polynomial in $2d$ non-commutative indeterminates.
The integration is performed with respect to the product measure $\d U_1 \cdots \d U_d$, where each $\d U_i$ represents the reference probability measure associated with a specific ensemble. Specifically, we consider $\d U_i$ to be either Haar-distributed matrices in $\U(N),\Ort(N), \Sp(N)$, or GUEs, GOEs, GSEs. (In the cases of those Gaussian ensembles, they are self-adjoint, so the dependence on $U_i^*$ in $V$ is equivalent to a dependence on $U_i$.)

Note that the free energy $F_N(V)$ is well-defined since $\Tr V$ is real-valued.
If one replaces $V$ by $tV$ for $t$ a real parameter, it is real analytic in $t$ for sufficiently small $t$.

Let $\mu_V^N$ be the probability measure 
defined by
\begin{align}
\d\mu_V^N (U_1,...,U_d)
:=\frac{1}{I_N (V)} \e^{N \Tr V(U_1,...,U_d,U_1^*,...,U_d^*)} \d U_1\cdots \d U_d.
\end{align}
For any self-adjoint polynomial $P$ in $2d$ non-commutative indeterminates, we define the expectation with respect to $\mu_V^N$ by
\begin{align}
\E_{\mu_V^N} [\tr P]
&:=\int
\tr P(U_1,...,U_d,U_1^*,...,U_d^*) \d \mu_V^N (U_1,...,U_d) \notag \\
&=\frac{\int \frac{1}{N} \Tr P \e^{N \Tr V} \d U_1\cdots \d U_d}{\int \e^{N \Tr V} \d U_1\cdots \d U_d}.
\end{align}

The asymptotic expansion of these matrix integrals has been studied for various ensembles when $V$ is a non-commutative polynomial. 
A fundamental property of these integrals is the existence of a $1/N$ expansion (or topological expansion) of the free energy $F_N (V)$. 
In the case of Gaussian matrices, the expansion was proved in \cite{MaurelSegala06}, \cite{GMS06_MR2249657}, by using the Schwinger-Dyson equations and combinatorial techniques. 
In the case of Haar-distributed matrices, the expansion was obtained relying on the Weingarten calculus as in \cite{Collins2003}, \cite{CS06_MR2217291}.
While stronger results establishing the analytic convergence of these series (i.e., a positive radius of convergence) have been obtained for polynomials \cite{CGMS09_MR2531371}, \cite{GN15_MR3331788}, \cite{Buc-dAlche2023}, we focus here on extending the existence of the topological expansion to the case of smooth potentials. 
In the following theorem, using the smooth asymptotic expansion theorems (e.g. Theorem \ref{thm:ExpansionSmoothGUE}, Theorem \ref{thm:ExpansionSmoothUnitary}), we establish that the asymptotic coefficients exist for smooth test functions, though we do not address the analytic convergence of the resulting formal series.

\begin{theorem}[Matrix integrals with smooth test functions]
Let $r$ be a fixed positive integer. We fix self-adjoint non-commutative polynomials $P_i \in \C \langle x_1,...,x_d,x_1^*,...,x_d^* \rangle \ (i=1,...,r)$ and $Q \in \C \langle x_1,...,x_d,x_1^*,...,x_d^* \rangle$.
For any smooth functions $h_1,...,h_r \in C^\infty (\R)$, (and in the Gaussian ensemble cases, we additionally assume boundedness), we define 
\begin{align*}
V=V_{\bm t}=t_1 h_1 (P_1)+\cdots +t_r h_r (P_r),
\end{align*}
for $\bm{t}=(t_1,...,t_r)$.
Then, for the matrix integrals defined above, the following holds:\\
(1)
The formal power series
\begin{align}
F_N (V)
=\sum_{\bm{k}\in \Z_{\geq 0}^r \setminus \{\bm{0}\}} a_{\bm{k}}^N \bm{t}^{\bm{k}},
\end{align}
is that for each $\bm{k} \in \Z_{\geq 0}^r \setminus \{\bm{0}\}$, the limit
\begin{align*}
\lim_{N \to \infty} a_{\bm{k}}^N
\end{align*}
exists and is finite, and it depends only on $V$.\\
(2)
For any smooth function $g \in C^\infty (\R)$, (and in the Gaussian ensemble cases, we additionally assume boundedness),the formal power series
\begin{align}
\E_{\mu_{V_{\bm t}}^N} [\tr g(Q)]
=\sum_{\bm{k}\in \Z_{\geq 0}^r} b_{\bm{k}}^N \bm{t}^{\bm{k}} ,
\end{align}
is that for each $\bm{k} \in \Z_{\geq 0}^r$, the limit
\begin{align*}
\lim_{N \to \infty} b_{\bm{k}}^N
\end{align*}
exists and is finite, and it depends only on $V$, $g$ and $Q$.
\end{theorem}

\begin{proof}
(1)
We set $X_i^N:=P_i(U_1^N,...,U_d^N,U_1^{N*},...,U_d^{N*})$ for $\bm{U}^N = (U_1^N,...,U_d^N)$.
Then, we have
\begin{align}
F_N (V)
&=\frac{1}{N^2} \log \E\left[\e^{N \Tr \left(t_1 h_1 (X_1^N)+\cdots +t_r h_r (X_r^N)\right)}\right] \notag \\
&=\sum_{\bm{k}\in \Z_{\geq 0}^r \setminus \{\bm{0}\}} \frac{1}{N^2 \bm{k}!} C_{\bm{k}}(N \Tr h_1(X_1^N),...,N \Tr h_r(X_r^N)) \bm{t}^{\bm{k}},
\end{align}
where we used the relation between cumulants and moment generating functions, \eqref{eq:CumulantGenerating_multi}, in the last equality. 
Note that the radius of convergence of the free energy $F_N(V)$ a priori depends on $N$, but the above expansion can be  considered as a formal power series in $\bm{t}$.
Thus for $\bm{k} \in \Z_{\geq 0}^r \setminus \{\bm{0}\}$, we may set
\begin{align}
a_{\bm{k}}^N :=\frac{1}{N^2 \bm{k}!} C_{\bm{k}}(N \Tr h_1(X^N),...,N \Tr h_r(X^N)),
\end{align}
and by Theorem \ref{thm:CumulantSmooth}, the limit as $N \to \infty$ exists and is finite.\\
(2)
We also set $Y^N:=Q(U_1^N,...,U_d^N,U_1^{N*},...,U_d^{N*})$.
Then, for fixed $N$, there exists a neighborhood of $\bm{t}=\bm{0}$ such that for any $\bm{t}$ in this neighborhood, we have 
\begin{align}
&\E_{\mu_{V_{\bm t}}^N} [\tr g(Q)] \notag \\
&=\frac{\E\left[\frac{1}{N} \Tr g(Y^N) \e^{N \Tr \left(t_1 h_1 (X_1^N)+\cdots +t_r h_r (X_r^N)\right)}\right]}{\E\left[\e^{N \Tr \left(t_1 h_1 (X_1^N)+\cdots +t_r h_r (X_r^N)\right)}\right]} \notag \\
&=\left.\frac{\partial}{\partial s}\right|_{s=0} F_N (s g(Q)+V) \notag \\
&=\left.\frac{\partial}{\partial s}\right|_{s=0} \sum_{(k_0,\bm{k})\in \Z_{\geq 0}^{r+1} \setminus \{\bm{0}\}} \frac{1}{N^2} C_{k_0, \bm{k}}(N \Tr g(Y^N),N \Tr h_1(X^N),...,N \Tr h_r(X^N)) \frac{s^{k_0} \bm{t}^{\bm{k}}}{k_0! \bm{k}!} \notag \\
&=\sum_{\bm{k}\in \Z_{\geq 0}^r} \frac{1}{N^2 \bm{k}!} C_{1,\bm{k}}(N \Tr g(Y^N), N \Tr h_1(X^N),...,N \Tr h_r(X^N)) \bm{t}^{\bm{k}},
\end{align}
as a formal power series, where we again used \eqref{eq:CumulantGenerating_multi} in the second to the last equality.
Thus for $\bm{k} \in \Z_{\geq 0}^r$, we may set
\begin{align}
b_{\bm{k}}^N =\frac{1}{N^2 \bm{k}!} C_{1,\bm{k}}(N \Tr g(Y^N), N \Tr h_1(X^N),...,N \Tr h_r(X^N)),
\end{align}
and again by Theorem \ref{thm:CumulantSmooth}, the limit as $N \to \infty$ exists and is finite.
\end{proof}

\begin{remark}[Regularization of formal matrix integrals]We conclude with an observation on the relevance of smooth test functions to formal matrix integrals appearing in combinatorics and theoretical physics. Consider, for instance, the cubic matrix model, formally given by the partition function:\begin{align*}I_N(\varepsilon) = \int \e^{-N \Tr\left(\frac{1}{2}X^2 + \varepsilon X^3\right)} \d X.\end{align*}It is well known that this integral diverges for any $\varepsilon \neq 0$ and does not exist in the Lebesgue sense, as the term $\varepsilon x^3$ dominates the Gaussian term $-x^2/2$ at infinity. Consequently, the formal power series in $\varepsilon$ (whose coefficients enumerate maps) cannot be directly interpreted as an asymptotic expansion of a well-defined integral. In their seminal work, Guionnet and Maurel-Segala \cite{GMS06_MR2249657} circumvented this difficulty by restricting the integration to a compact set of matrices (truncation) and proving that the asymptotic expansion is independent of the cutoff, provided it is sufficiently large. Our framework offers an alternative, ``smooth regularization'' approach. Instead of considering the unbounded polynomial potential $V(x) = x^3$, one may replace it with a smooth, bounded function $h \in C_b^\infty(\mathbb{R})$ such that $h(x) = x^3$ on a sufficiently large neighborhood of $[-2,2]$ (the support of the limiting spectral measure). With this substitution, the regularized integral\begin{align*}\tilde{I}_N(\varepsilon) = \int \e^{-N \Tr\left(\frac{1}{2}X^2 + \varepsilon h(X)\right)} \d X\end{align*}is well-defined in the Lebesgue sense for any fixed $N$, without requiring explicit domain truncation, as the Gaussian term dominates the bounded perturbation $h$. Theorem \ref{thm:CumulantSmooth} then ensures that the coefficients of the expansion of the free energy associated to $\tilde{I}_N(\varepsilon)$ admit well-defined limits as $N \to \infty$. Since $h$ coincides with the polynomial potential on the support of the spectrum, these limits coincide with the combinatorial quantities of interest.\end{remark}

\bibliography{ref} 

\begin{thebibliography}{CGVTvH25}

\bibitem[BD17]{BD17_MR3556288}
Jonathan Breuer and Maurice Duits.
\newblock Central limit theorems for biorthogonal ensembles and asymptotics of recurrence coefficients.
\newblock {\em J. Amer. Math. Soc.}, 30(1):27--66, 2017.

\bibitem[Bd24]{Buc-dAlche2023}
Thomas Buc-d'Alch{\'e}.
\newblock Topological expansion of unitary integrals and maps, 2024.
\newblock To appear in \emph{Ann. Inst. Henri Poincar\'e D}. \url{https://arxiv.org/abs/2304.12785}.

\bibitem[BH23]{BaoHe23_MR4677731}
Zhigang Bao and Yukun He.
\newblock Quantitative {CLT} for linear eigenvalue statistics of {W}igner matrices.
\newblock {\em Ann. Appl. Probab.}, 33(6B):5171--5207, 2023.

\bibitem[BIPZ78]{BIPZ78_MR471676}
E.~Br\'ezin, C.~Itzykson, G.~Parisi, and J.~B. Zuber.
\newblock Planar diagrams.
\newblock {\em Comm. Math. Phys.}, 59(1):35--51, 1978.

\bibitem[BP09]{BP09_MR2480549}
W\l{}odzimierz Bryc and Virgil Pierce.
\newblock Duality of real and quaternionic random matrices.
\newblock {\em Electron. J. Probab.}, 14:no. 17, 452--476, 2009.

\bibitem[CES23]{CES23_MR4551555}
Giorgio Cipolloni, L\'aszl\'o{} Erd\"os, and Dominik Schr\"oder.
\newblock Functional central limit theorems for {W}igner matrices.
\newblock {\em Ann. Appl. Probab.}, 33(1):447--489, 2023.

\bibitem[CGMS09]{CGMS09_MR2531371}
Beno\^it Collins, Alice Guionnet, and Edouard Maurel-Segala.
\newblock Asymptotics of unitary and orthogonal matrix integrals.
\newblock {\em Adv. Math.}, 222(1):172--215, 2009.

\bibitem[CGVTvH25]{vH_newapproach1}
Chi-Fang Chen, Jorge Garza-Vargas, Joel~A. Tropp, and Ramon van Handel.
\newblock A new approach to strong convergence, 2025.
\newblock To appear in \emph{Ann. of Math. (2)}. \url{https://arxiv.org/abs/2405.16026}.

\bibitem[CGVvH24]{VH_newapproach2}
Chi-Fang Chen, Jorge Garza-Vargas, and Ramon van Handel.
\newblock A new approach to strong convergence ii. the classical ensembles, 2024.

\bibitem[Cho80]{Choi1980_MR590864}
Man~Duen Choi.
\newblock The full {$C\sp{\ast} $}-algebra of the free group on two generators.
\newblock {\em Pacific J. Math.}, 87(1):41--48, 1980.

\bibitem[CJ24]{CJ24_MR4718383}
Klara Courteaut and Kurt Johansson.
\newblock Multivariate normal approximation for traces of orthogonal and symplectic matrices.
\newblock {\em Ann. Inst. Henri Poincar\'e{} Probab. Stat.}, 60(1):312--342, 2024.

\bibitem[CMSS07]{CMSS07}
Beno\^it Collins, James~A. Mingo, Piotr \'Sniady, and Roland Speicher.
\newblock Second order freeness and fluctuations of random matrices. {III}. {H}igher order freeness and free cumulants.
\newblock {\em Doc. Math.}, 12:1--70, 2007.

\bibitem[Col03]{Collins2003}
Beno^^c3^^aet Collins.
\newblock Moments and cumulants of polynomial random variables on unitarygroups, the itzykson-zuber integral, and free probability.
\newblock {\em International Mathematics Research Notices}, 2003(17):953--982, 2003.

\bibitem[CS06]{CS06_MR2217291}
Beno\^it Collins and Piotr \'Sniady.
\newblock Integration with respect to the {H}aar measure on unitary, orthogonal and symplectic group.
\newblock {\em Comm. Math. Phys.}, 264(3):773--795, 2006.

\bibitem[DM23]{DM23_MR4552697}
Mario Diaz and James~A. Mingo.
\newblock On the analytic structure of second-order non-commutative probability spaces and functions of bounded {F}r\'echet variation.
\newblock {\em Random Matrices Theory Appl.}, 12(1):Paper No. 2250044, 26, 2023.

\bibitem[EM03]{EM03_MR1953782}
N.~M. Ercolani and K.~D. T.-R. McLaughlin.
\newblock Asymptotics of the partition function for random matrices via {R}iemann-{H}ilbert techniques and applications to graphical enumeration.
\newblock {\em Int. Math. Res. Not.}, (14):755--820, 2003.

\bibitem[GMS06]{GMS06_MR2249657}
Alice Guionnet and Edouard Maurel-Segala.
\newblock Combinatorial aspects of matrix models.
\newblock {\em ALEA Lat. Am. J. Probab. Math. Stat.}, 1:241--279, 2006.

\bibitem[GN15]{GN15_MR3331788}
Alice Guionnet and Jonathan Novak.
\newblock Asymptotics of unitary multimatrix models: the {S}chwinger-{D}yson lattice and topological recursion.
\newblock {\em J. Funct. Anal.}, 268(10):2851--2905, 2015.

\bibitem[Gui02]{Guionnet2002_MR1899457}
Alice Guionnet.
\newblock Large deviations upper bounds and central limit theorems for non-commutative functionals of {G}aussian large random matrices.
\newblock {\em Ann. Inst. H. Poincar\'e{} Probab. Statist.}, 38(3):341--384, 2002.

\bibitem[HTr12]{HT12_MR2922846}
Uffe Haagerup and Steen Thorbj\o~rnsen.
\newblock Asymptotic expansions for the {G}aussian unitary ensemble.
\newblock {\em Infin. Dimens. Anal. Quantum Probab. Relat. Top.}, 15(1):1250003, 41, 2012.

\bibitem[HZ86]{HarerZagier86_MR848681}
J.~Harer and D.~Zagier.
\newblock The {E}uler characteristic of the moduli space of curves.
\newblock {\em Invent. Math.}, 85(3):457--485, 1986.

\bibitem[JL21]{JL21_MR4348683}
Kurt Johansson and Gaultier Lambert.
\newblock Multivariate normal approximation for traces of random unitary matrices.
\newblock {\em Ann. Probab.}, 49(6):2961--3010, 2021.

\bibitem[Joh98]{Johansson1998_MR1487983}
Kurt Johansson.
\newblock On fluctuations of eigenvalues of random {H}ermitian matrices.
\newblock {\em Duke Math. J.}, 91(1):151--204, 1998.

\bibitem[LP09]{LP09_MR2561434}
A.~Lytova and L.~Pastur.
\newblock Central limit theorem for linear eigenvalue statistics of random matrices with independent entries.
\newblock {\em Ann. Probab.}, 37(5):1778--1840, 2009.

\bibitem[Mat13]{Matsumoto13_MR3077830}
Sho Matsumoto.
\newblock Weingarten calculus for matrix ensembles associated with compact symmetric spaces.
\newblock {\em Random Matrices Theory Appl.}, 2(2):1350001, 26, 2013.

\bibitem[MMPS22]{CMPS22_MR4414193}
Camile Male, James~A. Mingo, Sandrine P\'ech\'e, and Roland Speicher.
\newblock Joint global fluctuations of complex {W}igner and deterministic matrices.
\newblock {\em Random Matrices Theory Appl.}, 11(2):Paper No. 2250015, 46, 2022.

\bibitem[MP13]{MP13_MR3098908}
James~A. Mingo and Mihai Popa.
\newblock Real second order freeness and {H}aar orthogonal matrices.
\newblock {\em J. Math. Phys.}, 54(5):051701, 35, 2013.

\bibitem[MP24]{MP24_MR4700374}
Michael Magee and Doron Puder.
\newblock Matrix group integrals, surfaces, and mapping class groups {II}: {${\rm O}(n)$} and {${\rm Sp}(n)$}.
\newblock {\em Math. Ann.}, 388(2):1437--1494, 2024.

\bibitem[MS06a]{MaurelSegala06}
Edouard Maurel-Segala.
\newblock High order expansion of matrix models and enumeration of maps.
\newblock 09 2006.

\bibitem[MS06b]{MS06}
James~A. Mingo and Roland Speicher.
\newblock Second order freeness and fluctuations of random matrices. {I}. {G}aussian and {W}ishart matrices and cyclic {F}ock spaces.
\newblock {\em J. Funct. Anal.}, 235(1):226--270, 2006.

\bibitem[MS06c]{MS06_MR2216446}
James~A. Mingo and Roland Speicher.
\newblock Second order freeness and fluctuations of random matrices. {I}. {G}aussian and {W}ishart matrices and cyclic {F}ock spaces.
\newblock {\em J. Funct. Anal.}, 235(1):226--270, 2006.

\bibitem[MS17]{MR3585560_MingoSpeicher17}
James~A. Mingo and Roland Speicher.
\newblock {\em Free probability and random matrices}, volume~35 of {\em Fields Institute Monographs}.
\newblock Springer, New York; Fields Institute for Research in Mathematical Sciences, Toronto, ON, 2017.

\bibitem[MSS07]{MSS07_MR2294222}
James~A. Mingo, Piotr \'Sniady, and Roland Speicher.
\newblock Second order freeness and fluctuations of random matrices. {II}. {U}nitary random matrices.
\newblock {\em Adv. Math.}, 209(1):212--240, 2007.

\bibitem[NS06]{NS06_MR2266879}
Alexandru Nica and Roland Speicher.
\newblock {\em Lectures on the combinatorics of free probability}, volume 335 of {\em London Mathematical Society Lecture Note Series}.
\newblock Cambridge University Press, Cambridge, 2006.

\bibitem[Par23a]{parraud2023_unitary}
F^^c3^^a9lix Parraud.
\newblock Asymptotic expansion of smooth functions in deterministic and iid haar unitary matrices, and application to tensor products of matrices, 2023.

\bibitem[Par23b]{Parraud23_MR4567374}
F\'elix Parraud.
\newblock Asymptotic expansion of smooth functions in polynomials in deterministic matrices and iid {GUE} matrices.
\newblock {\em Comm. Math. Phys.}, 399(1):249--294, 2023.

\bibitem[Red14]{Redelmeier14_MR3217665}
Catherine Emily~Iska Redelmeier.
\newblock Real second-order freeness and the asymptotic real second-order freeness of several real matrix models.
\newblock {\em Int. Math. Res. Not. IMRN}, (12):3353--3395, 2014.

\bibitem[Rek24]{Reker24_MR4842843}
Jana Reker.
\newblock Multi-point functional central limit theorem for {W}igner matrices.
\newblock {\em Electron. J. Probab.}, 29:Paper No. 191, 49, 2024.

\bibitem[Rot64]{Rota64_MR174487}
Gian-Carlo Rota.
\newblock On the foundations of combinatorial theory. {I}. {T}heory of {M}\"obius functions.
\newblock {\em Z. Wahrscheinlichkeitstheorie und Verw. Gebiete}, 2:340--368, 1964.

\bibitem[RS00]{RotaShen00_MR1779783}
Gian-Carlo Rota and Jianhong Shen.
\newblock On the combinatorics of cumulants.
\newblock volume~91, pages 283--304. 2000.
\newblock In memory of Gian-Carlo Rota.

\bibitem[Sch05]{Schultz05_MR2117954}
Hanne Schultz.
\newblock Non-commutative polynomials of independent {G}aussian random matrices. {T}he real and symplectic cases.
\newblock {\em Probab. Theory Related Fields}, 131(2):261--309, 2005.

\bibitem[Shc11]{Shcherbina11_MR2829615}
M.~Shcherbina.
\newblock Central limit theorem for linear eigenvalue statistics of the {W}igner and sample covariance random matrices.
\newblock {\em J. Math. Phys. Anal. Geom.}, 7(2):176--192, 197, 199, 2011.

\bibitem[Sil90]{Silverstein90_MR1062064}
Jack~W. Silverstein.
\newblock Weak convergence of random functions defined by the eigenvectors of sample covariance matrices.
\newblock {\em Ann. Probab.}, 18(3):1174--1194, 1990.

\bibitem[Sil22]{Silverstein22_MR4480828}
Jack~W. Silverstein.
\newblock Weak convergence of a collection of random functions defined by the eigenvectors of large dimensional random matrices.
\newblock {\em Random Matrices Theory Appl.}, 11(4):Paper No. 2250033, 33, 2022.

\bibitem[Sos00]{Soshnikov00_MR1797877}
Alexander Soshnikov.
\newblock The central limit theorem for local linear statistics in classical compact groups and related combinatorial identities.
\newblock {\em Ann. Probab.}, 28(3):1353--1370, 2000.

\bibitem[Sta97]{Stanley97_MR1442260}
Richard~P. Stanley.
\newblock {\em Enumerative combinatorics. {V}ol. 1}, volume~49 of {\em Cambridge Studies in Advanced Mathematics}.
\newblock Cambridge University Press, Cambridge, 1997.
\newblock With a foreword by Gian-Carlo Rota, Corrected reprint of the 1986 original.

\bibitem[SW13]{SW13_MR3116567}
Philippe Sosoe and Percy Wong.
\newblock Regularity conditions in the {CLT} for linear eigenvalue statistics of {W}igner matrices.
\newblock {\em Adv. Math.}, 249:37--87, 2013.

\bibitem[tH74]{tHooft74_MR413809}
G.~'t~Hooft.
\newblock Magnetic monopoles in unified gauge theories.
\newblock {\em Nuclear Phys.}, B79:276--284, 1974.

\bibitem[Voi91]{Voiculescu91_MR1094052}
Dan Voiculescu.
\newblock Limit laws for random matrices and free products.
\newblock {\em Invent. Math.}, 104(1):201--220, 1991.

\bibitem[Wei78]{Weingarten}
Don Weingarten.
\newblock Asymptotic behavior of group integrals in the limit of infinite rank.
\newblock {\em Journal of Mathematical Physics}, 19(5):999--1001, 05 1978.

\bibitem[Zvo97]{Zvonkin97_MR1492512}
A.~Zvonkin.
\newblock Matrix integrals and map enumeration: an accessible introduction.
\newblock volume~26, pages 281--304. 1997.
\newblock Combinatorics and physics (Marseilles, 1995).

\end{thebibliography}
\bibliographystyle{alpha}

\end{document}